\documentclass[11pt]{amsart}

\usepackage[margin=1in, marginparwidth=2cm]{geometry} % Standard margins for preprints
\usepackage{changepage} % For adjusting margins locally
\usepackage{enumitem}   % Customized lists (e.g., (i), (ii) instead of 1, 2)
\usepackage{appendix}   % Better control over appendix formatting
\usepackage{microtype}  % Improves spacing and justification (especially for math)

\usepackage{mathtools}  % Extension of amsmath, fixes bugs and adds features
\usepackage{amssymb}    % Standard symbols (blackboard bold, etc.)
\usepackage{mathrsfs}   % For \mathscr (fancy script fonts)
\usepackage{diffcoeff}  % Easy derivative notation
\usepackage{thmtools}
\usepackage{thm-restate}

\usepackage{float}      % Improved interface for floating objects (figures, tables)
\usepackage{graphicx}   % Required for inserting images
\usepackage[all]{xy}    % Legacy diagram package (keep if you have old code)

\usepackage{tikz}
\usepackage{tikz-cd}    % The modern standard for commutative diagrams

\usetikzlibrary{
    patterns,
    arrows.meta,
    bending,
    calc,
    intersections
}

\usepackage{cite}       % Better citation handling
\usepackage[colorlinks=true, linkcolor=blue, citecolor=green, urlcolor=black]{hyperref} 
\usepackage[capitalize,nameinlink]{cleveref} % Smart referencing (e.g., "Theorem 1.2" instead of just "1.2")

\usepackage{marginnote}
\usepackage[colorinlistoftodos, textsize=tiny]{todonotes}

\theoremstyle{plain} % Italic text, bold title
\newtheorem{thm}{Theorem}[section]
\newtheorem{lemma}[thm]{Lemma}
\newtheorem{cor}[thm]{Corollary}
\newtheorem{prop}[thm]{Proposition}

\newtheorem{ques}[thm]{Question}

\theoremstyle{definition} % Roman text, bold title
\newtheorem{defn}[thm]{Definition}

\theoremstyle{remark} % Roman text, italic title
\newtheorem*{rmk}{Remark}

\newcommand{\mb}{\mathbb}

\newcommand{\mc}{\mathcal}

\newcommand{\Z}{\mathbb{Z}}
\newcommand{\R}{\mathbb{R}}
\newcommand{\C}{\mathbb{C}}

\newcommand{\inv}{^{-1}} 
\newcommand{\Flat}{\text{Flat}(\Sigma)}

\DeclarePairedDelimiter{\paren}{(}{)}           % For parentheses
\DeclarePairedDelimiterX{\innerproduct}[2]{\langle}{\rangle}{#1, #2}
\DeclarePairedDelimiter{\abs}{\lvert}{\rvert}   % Absolute value
\DeclarePairedDelimiterX{\set}[2]{\{}{\}}{#1 \mathrel{}\delimsize\vert\mathrel{} #2} % Set builder

\DeclareMathOperator{\UE}{UE}
\DeclareMathOperator{\PUE}{PUE}
\DeclareMathOperator{\CAT}{CAT(0)}
\DeclareMathOperator{\Len}{Len}
\DeclareMathOperator{\EL}{EL}     % extremal length
\DeclareMathOperator{\NPC}{NPC}   % nonpositively curved metrics with cone points
\DeclareMathOperator{\Area}{Area}

\DeclareMathOperator{\Carr}{Carr}
\DeclareMathOperator{\Curr}{Curr}
\DeclareMathOperator{\diag}{diag}

\DeclareMathOperator{\SL}{SL}
\DeclareMathOperator{\SO}{SO}

\newcommand{\til}{\widetilde}

\newcommand{\wideS}{\widetilde{\Sigma}}

\DeclareMathOperator{\cyl}{cyl}

\title{Thurston's Asymmetric Metric for the Space of Flat Metrics}
\author{Jiajun Shi}
\address{Max Planck Institute for Mathematics in the Sciences, 04103 Leipzig, Germany}
\email{jiajun.shi@mis.mpg.de}

\begin{document}

\begin{abstract}
    Thurston introduced in his seminal work \cite{thurston1998minimal} an asymmetric metric on Teichm\"uller space by the ratio of simple closed curve length. In this paper, we generalize the idea and define an asymmetric metric on the space of unit-area flat metrics coming from half-translation structures on a closed surface. We also discuss two different topologies coming from the asymmetry.
\end{abstract}

\maketitle

\tableofcontents

\section{Introduction}

Fix a finite-type surface $\Sigma$ with negative Euler characteristic. Given a locally $\CAT$ complete metric $g$ on $\Sigma$, each essential nonperipheral free homotopy class of closed curves admits geodesic representatives of minimal length, and all such representatives have the same length. We denote this length by $\ell_g(\gamma)$.

Thurston defined in \cite{thurston1998minimal} the following asymmetric distance on Teichm\"uller space:
\[
K(g_1,g_2)
=
\log \sup_{\gamma\in \pi_1(\Sigma)}
\frac{\ell_{g_2}(\gamma)}{\ell_{g_1}(\gamma)},
\qquad g_i\in \mc T(\Sigma).
\]
He proved that $K(g_1,g_2)\ge 0$, with equality if and only if $g_1=g_2$ as points of $\mc T(\Sigma)$. Equivalently, for any two distinct points $g_1,g_2\in\mc T(\Sigma)$, there exists a simple closed curve $\gamma$ such that
\[
\ell_{g_2}(\gamma)>\ell_{g_1}(\gamma).
\]
Since $K$ also satisfies the triangle inequality, it defines an asymmetric metric on $\mc T(\Sigma)$, now called \emph{Thurston's metric}.

The nondegeneracy of $K$ can be viewed as a form of \emph{volume rigidity}. Let $\mc M$ be a space of metrics on $\Sigma$, considered up to isometry isotopic to the identity. We say that $\mc M$ is volume rigid if, for any two metrics $g_1,g_2\in\mc M$ with
\[
\Area(g_1)=\Area(g_2),
\]
the inequality
\[
\ell_{g_2}(\gamma)\ge \ell_{g_1}(\gamma)
\]
for every essential nonperipheral closed curve $\gamma$ implies that $g_1=g_2$. Croke--Dairbekov proved volume rigidity for negatively curved metrics in \cite{croke2004lengths}. Their proof uses a result of Lopes--Thieullen \cite{lopes2005sub} on sub-actions for the geodesic flow. Constantine--Shrestha--Wu generalized the sub-action argument to locally $\mathrm{CAT}(-1)$ spaces and proved volume rigidity for locally $\mathrm{CAT}(-1)$ metrics in \cite{CONSTANTINE_SHRESTHA_WU_2025}.

Another important rigidity property is \emph{marked length spectrum rigidity}. A space $\mc M$ of metrics is marked length spectrum rigid if, for any two metrics $g_1,g_2\in\mc M$, the equality
\[
\ell_{g_1}(\gamma)=\ell_{g_2}(\gamma)
\]
for every essential nonperipheral closed curve $\gamma$ implies that $g_1=g_2$. If all metrics in $\mc M$ have the same area, then volume rigidity implies marked length spectrum rigidity. It is natural to ask in which settings the two rigidity properties are equivalent.

In this paper, we focus on a family of $\CAT$ metrics arising from half-translation structures on $\Sigma$. Such a metric is locally Euclidean away from finitely many cone singularities, and we call it a \emph{flat metric}. Let $\Flat$ denote the space of unit-area flat metrics on $\Sigma$, considered up to isometry isotopic to the identity. Duchin--Leininger--Rafi proved in \cite{duchin2010length} that $\Flat$ is \emph{simple marked length spectrum rigid}: if two metrics in $\Flat$ have the same lengths for all simple closed curves, then they are equal.

Following the strategy of \cite{duchin2010length}, we prove the following stronger volume-rigidity statement for $\Flat$.

\begin{thm}\label{main theorem}
    Given $q_1,q_2\in \Flat$, suppose that
    \[
    \ell_{q_2}(\gamma)\ge \ell_{q_1}(\gamma)
    \]
    for every essential nonperipheral simple closed curve $\gamma$. Then $q_1=q_2$.
\end{thm}

Let $\mc S$ denote the set of homotopy classes of essential nonperipheral simple closed curves. In analogy with Thurston's construction, we define
\begin{equation}\label{equ:K}
    K(q_1,q_2)
    =
    \log\sup_{\gamma\in\mc S}
    \frac{\ell_{q_2}(\gamma)}{\ell_{q_1}(\gamma)},
    \qquad q_i\in\Flat.
\end{equation}
There are two differences from Thurston's original definition. First, we take the supremum only over simple closed curves. This gives a smaller supremum than the one over all closed curves, so nondegeneracy for this version implies nondegeneracy for the version using all closed curves. Second, we work in the unit-area space $\Flat$. This normalization is necessary for flat metrics, since homothetic rescaling changes all lengths by the same factor. Without fixing the area, the corresponding asymmetric distance could take negative values. In the hyperbolic setting, the normalization is automatic by the Gauss--Bonnet theorem, and the supremum over simple closed curves agrees with Thurston's original definition by \cite[Proposition 3.5]{thurston1998minimal}.

An immediate consequence of Theorem \ref{main theorem} is the nondegeneracy of $K$.

\begin{cor}\label{cor:nondegenerate}
    For any $q_1,q_2\in\Flat$, we have $K(q_1,q_2)\ge0$, with equality if and only if $q_1=q_2$.
\end{cor}

The function $K$ is generally not symmetric. It therefore gives rise to two natural notions of balls:
\[
    B_q^l(R)=\{t\in\Flat \mid K(t,q)\le R\},
\]
and
\[
    B_q^r(R)=\{t\in\Flat \mid K(q,t)\le R\}.
\]
On the other hand, Duchin--Leininger--Rafi embedded $\Flat$ into the space of geodesic currents $\Curr(\Sigma)$ in \cite[Theorem 4]{duchin2010length}. Thus $\Flat$ also carries the topology inherited from the space of geodesic currents. We compare these topologies and obtain the following result.

\begin{thm}\label{compare topology}
    The topology on $\Flat$ generated by the left balls $B_q^l(R)$ coincides with the topology inherited from geodesic currents. Moreover, convergence in the topology of geodesic currents implies convergence with respect to the right distance: if $q_n\to q$ in the topology of geodesic currents, then $K(q,q_n)\to 0$. However, the right distance has different large-scale behavior: for each fixed $q_0\in\Flat$, the function $K(q_0,\cdot)$ is not proper.
\end{thm}

In addition to proving the nondegeneracy of $K$ via volume rigidity, we also study closed curves whose lengths differ in different flat metrics, which serve as candidates realizing the supremum. We focus on the case where the two metrics lie in the same conformal class. In this setting, extremal length provides a natural way to produce many examples.

\begin{prop}
    Let $q$ and $q'$ be flat metrics in the same conformal class with the same area, and suppose that they are distinct as points of $\Flat$. Let $\mc F$ be the measured foliation of $q$ in direction $\theta$. Then every simple closed curve $\gamma$ whose projective class is sufficiently close to $[\mc F]$ in $\mc{PMF}(\Sigma)$ satisfies
    \[
        \ell_{q'}(\gamma)<\ell_q(\gamma).
    \]
\end{prop}

In fact, there is a notion of length for measured foliations, which can be interpreted as the intersection number between the corresponding measured foliation and the Liouville current of the metric. The proposition above is a consequence of a more general statement: under the same assumptions, the length of $\mc F$ with respect to $q'$ is strictly smaller than its length with respect to $q$. See Corollary \ref{c:foliation_length_conformal} for the precise statement.

Bonahon \cite[Theorem 19]{bonahon1988geometry} proved that, for hyperbolic metrics $h,h'$ and their Liouville currents $L_h,L_{h'}$, the intersection number satisfies
\[
    i(L_h,L_h)\le i(L_h,L_{h'}),
\]
with equality if and only if $h=h'$. As a consequence of the result above, we obtain the following statement, which contrasts with the hyperbolic case and is of independent interest.

\begin{cor}
    Given a Riemann surface $X$ and any two quadratic differentials $q,q'\in Q^1(X)$, let $L_q$ and $L_{q'}$ be the corresponding Liouville currents. Then
    \[
        i(L_q,L_q)\ge i(L_q,L_{q'}),
    \]
    with equality if and only if $q=cq'$ for some $c\in\C$ with $\abs{c}=1$.
\end{cor}

We also mention related work of Sa\u{g}lam--Papadopoulos \cite{sauglam2024thurston}. They considered the space of unit-area flat metrics with a fixed quadrangulation and modified the definition of $K$ by taking ratios of edge lengths in the quadrangulation. They proved that the resulting function is nondegenerate. Although their construction concerns a restricted subset of $\Flat$, it has a geometric interpretation similar to \cite[Theorem 8.5]{thurston1998minimal}: the function agrees with the infimum of Lipschitz constants of maps that fix the vertices of the quadrangulation and are homotopic to the identity.

\subsection{Main ideas and structure}

We briefly describe the proof of Theorem \ref{main theorem}. The first step is to extract information about cylinder curves; see Section \ref{sec:geodesic} for the definitions of cylinder curves and singular closed geodesics. From the assumed inequalities for simple closed curves, we obtain the corresponding inequalities for measured foliations by using the density of weighted simple closed curves. Then, using an argument based on intersections of Liouville currents, in the spirit of \cite{croke2004lengths}, the equality of areas and the length inequalities force cylinder curves in either metric to have the same length.

The next step is to compare the cylinder curves of $q_1$ and $q_2$. We use a technique from \cite{duchin2010length} involving Dehn twists. More precisely, the effect of a Dehn twist on length distinguishes cylinder curves from singular closed geodesics. If $\alpha$ is a cylinder curve, then for any curve $\beta$ intersecting $\alpha$ transversely, the Dehn-twisted curve $T_\alpha\beta$ satisfies a strict length inequality. If $\alpha$ is not a cylinder curve, then one can find a curve $\beta$ for which the corresponding inequality becomes an equality. See Section \ref{sec: Dehn twist} for the detailed statement, and see \cite[\S 3]{duchin2010length} for the original argument.

\begin{figure}[htbp]
    \centering
    \includegraphics[scale=0.7]{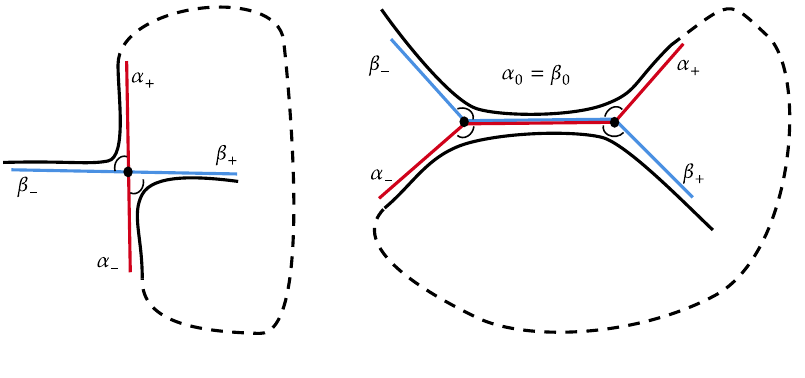}
    \caption{Two possible configurations for $T_\alpha\beta$ when $i(\alpha,\beta)=1$. Whether the black curve is geodesic depends on the intersection angle.}
    \label{fig: Dehn twist intro}
\end{figure}

Using this Dehn-twist criterion, we prove that every cylinder curve for $q_2$ is also a cylinder curve for $q_1$. Suppose otherwise. Then there exists a curve $\alpha$ that is a cylinder curve for $q_2$ but not for $q_1$. The key technical point is to find another simple closed curve $\beta$ such that
\begin{itemize}
    \item $\beta$ is a cylinder curve for $q_2$;
    \item in the metric $q_1$, the curve $\beta$ intersects $\alpha$ in the configuration that realizes equality in the Dehn-twist length estimate.
\end{itemize}
The existence of such a curve $\beta$ follows from the density of cylinder curves in the phase space; see Lemma \ref{l:technical}. The Dehn-twist length estimates then give
\[
\ell_{q_2}(T_\alpha\beta)<\ell_{q_1}(T_\alpha\beta),
\]
contradicting the hypothesis of Theorem \ref{main theorem}.

Once we know the inclusion of cylinder sets, we adapt an argument from \cite[Lemma 22]{duchin2010length}. This produces a pair of measured foliations that determine the $\SL_2\R$-orbits of both $q_1$ and $q_2$. Finally, a singular-value-decomposition argument (see Section \ref{sec:ex}) shows that the two flat metrics must coincide.

The structure of the paper is as follows.
\begin{itemize}
    \item Section \ref{preliminaries} recalls the basic definitions and notation for half-translation surfaces, quadratic differentials, measured foliations, and geodesic currents.
    \item Section \ref{examples} gives explicit examples in which one can find curves realizing strict length inequalities. It also discusses an example involving surfaces with boundary.
    \item Section \ref{proof} proves Theorem \ref{main theorem}.
    \item Section \ref{topology} compares the topology coming from geodesic currents with the topologies induced by the asymmetric distance $K$.
    \item Section \ref{ex extremal length} studies the special case of flat metrics in the same conformal class.
    \item Appendix \ref{appendix} discusses the notion of length for measured foliations and gives a new formula for the extremal length of measured foliations.
\end{itemize}

\subsection*{Acknowledgements}

The author would like to thank his advisor Anna Wienhard for suggesting this problem and for her kind support. The author also thanks Xian Dai, Alex Wright, Christopher Leininger, James Farre, \c{C}a\u{g}lar Uyanik, Beatrice Pozzetti, Dylan Thurston and Didac Martinez-Granado for helpful conversations.  The author was partially supported by the Deutsche Forschungsgemeinschaft (DFG, German Research Foundation), Project-ID 281071066, TRR 191.

\section{Preliminaries}\label{preliminaries}

Throughout the paper, $\Sigma$ denotes a finite-type surface with negative Euler characteristic and without boundary components. We regard $\Sigma$ as the complement of a finite set $\mc E$ of marked points in a closed surface $\widehat{\Sigma}$:
\[
\Sigma=\widehat{\Sigma}\setminus\mc E.
\]
By uniformization, $\Sigma$ admits a complete finite-area hyperbolic metric.

This section fixes the terminology and notation used in the remainder of the paper. We first introduce half-translation surfaces from the viewpoints of quadratic differentials and geometric structures. We then describe their associated flat metrics, as well as the relevant properties of geodesics and of the universal cover equipped with its $\CAT$ metric completion.

Next, we review measured foliations and the directional measured foliations naturally associated with a half-translation surface. Finally, we recall the theory of geodesic currents, first for closed surfaces following \cite{bonahon1988geometry}, and then for finite-type surfaces following \cite{bonahonBoutsVarietesHyperboliques1986,duchin2010length,burgerCurrentsSystolesCompactifications2021}.

\subsection{Half-translation surface}

In this subsection, we recall two equivalent descriptions of half-translation surfaces. We first give the analytic description in terms of quadratic differentials. Let $X$ be a complex structure on the closed surface $\widehat{\Sigma}$. Thus $X$ is given by a maximal atlas whose charts take values in $\C$ and whose transition maps are holomorphic. We shall also use $X$ to denote the corresponding Riemann surface. Let $T^*X$ denote its holomorphic cotangent bundle.

\begin{defn}
A \emph{quadratic differential} on $\Sigma$ is a pair consisting of a complex structure $X$ on $\widehat{\Sigma}$ and a nonzero meromorphic section $q$ of $T^*X\otimes T^*X$ such that
\begin{itemize}
\item $q$ is holomorphic on $\Sigma=\widehat{\Sigma}\setminus \mc E$;
\item $q$ has at most simple poles at the marked points in $\mc E$.
\end{itemize}
\end{defn}

Locally, a quadratic differential has the form $f(z)\,dz^2$, where $z$ is a local coordinate for the complex structure and $f(z)$ is meromorphic. Away from the marked points, the function $f$ is holomorphic. The following lemma states that one can choose local coordinates in which the quadratic differential has a standard form.

\begin{lemma}[{\cite[\S 6]{strebel1984quadratic}}]\label{l:standard_chart}
Let $q$ be a quadratic differential on $\Sigma$. Let $\operatorname{Zero}(q)$ denote the set of zeros of $q$. Then there exists an atlas ${(U_i,w_i)}$ for $\widehat{\Sigma}$ such that:
\begin{itemize}
\item each $U_i$ contains at most one point of $\operatorname{Zero}(q)\cup \mc E$;
\item if $U_i$ contains no point of $\operatorname{Zero}(q)\cup \mc E$, then $q=dw_i^2$ on $U_i$;
\item if $U_i$ contains a point $p\in \operatorname{Zero}(q)$, then
\[
q=w_i^k\,dw_i^2
\]
on $U_i\setminus{p}$, where $w_i(p)=0$ and $k=\operatorname{ord}_p(q)\ge 3$.
\item if $U_i$ contains a point $p\in \mc E$, then
\[
q=w_i^k\,dw_i^2
\]
on $U_i\setminus{p}$, where $w_i(p)=0$ and $k=\operatorname{ord}_p(q)\ge -1$.
\end{itemize}
\end{lemma}

Such coordinates $w_i$ are called \emph{natural coordinates}, and the corresponding neighborhoods $U_i$ are called \emph{standard neighborhoods}.

Quadratic differentials are equivalent to certain geometric structures on the surface. We recall this second description next. Let $M$ be a manifold, let $X$ be a model space, and let $G$ be a group acting on $X$ by homeomorphisms. A $(G,X)$-structure on $M$ is an atlas whose charts are homeomorphisms onto open subsets of $X$ and whose transition maps are restrictions of elements of $G$.

\begin{defn}\label{def:geometric_structure}
A \emph{half-translation structure} on $\Sigma=\widehat{\Sigma}\setminus\mc E$ is given by the following data:
\begin{itemize}
\item a finite set of points $P\subset \Sigma$;
\item an $(\R^2\rtimes \Z/2\Z,\R^2)$-structure on $\Sigma\setminus P$, where $\R^2\rtimes \Z/2\Z\subset \operatorname{Isom}^+(\R^2)$ is the semidirect product of the translation group $\R^2$ and the group generated by rotation by $\pi$;
\item for each point $p\in P$, the corresponding point in the metric completion has a neighborhood isometric to a Euclidean cone with cone angle $k\pi$ for some integer $k\ge 3$;
\item for each point $p\in \mc E$, the corresponding point in the metric completion has a neighborhood isometric to a Euclidean cone with cone angle $k\pi$ for some integer $k\ge 1$.
\end{itemize}
The points in $P$ are called \emph{singular points} or \emph{cone points} in $\Sigma$, and the points in $\Sigma\setminus P$ are called \emph{regular points}.
\end{defn}

\begin{prop}[{\cite[\S 1.8]{masur2002rational}}]
The two descriptions above are equivalent. More precisely, a half-translation structure on $\Sigma$ determines a complex structure on $\widehat{\Sigma}$ and a corresponding quadratic differential on $\Sigma$. Conversely, a quadratic differential on $\Sigma$ determines a half-translation structure. If $p\in P\cup\mc E$, then the cone angle at $p$ is
\[
\bigl(\operatorname{ord}_p(q)+2\bigr)\pi.
\]
In particular, the points of $P$ correspond precisely to the zeros of $q$ lying in $\Sigma$.
\end{prop}

A surface equipped with a half-translation structure is called a \emph{half-translation surface}. By the preceding proposition, a half-translation surface can equivalently be viewed as a Riemann surface equipped with a quadratic differential.

\subsection{Flat metrics and quadratic differentials}\label{intro metric}

Let $\Sigma$ be a half-translation surface. The half-translation structure naturally induces a metric on $\Sigma \setminus P$, inherited from the Euclidean metric on the local $\mathbb{R}^2$ charts, since the transition maps preserve the Euclidean metric. Alternatively, this metric can be constructed using natural coordinates; see \cite{strebel1984quadratic}. Such a metric is a special case of a \emph{flat cone metric}, which is locally Euclidean except at finitely many points, with each exceptional point having a neighborhood isometric to a Euclidean cone. We refer to metrics arising from half-translation structures as \emph{flat metrics} or \emph{half-translation metrics}. For the precise difference between half-translation metrics and general flat cone metrics, see \cite[Theorem 2.6]{shiLiouvilleCurrentHolomorphic2024}.

The notions of length and area are well-defined for these metrics. These quantities can also be computed directly from the quadratic differential. Let $q$ be a quadratic differential, and write locally $q=f(z)\,dz^2$. If $\gamma$ is a piecewise $C^1$ curve, its length with respect to $q$ is given by
\[
    \Len_q(\gamma)=\int_\gamma \sqrt{\abs{f(z)}}\,\abs{dz}
\]
and the total area is
\begin{equation}\label{equ:area}
    \Area(q)=\int_X \abs{f(z)}\,dx\,dy.
\end{equation}
For further details, see \cite[\S 5.3]{strebel1984quadratic}.

Let $\Flat$ denote the space of unit-area half-translation metrics on $\Sigma$ up to isometries isotopic to the identity, and let $Q(\Sigma)$ be the space of quadratic differentials on $\Sigma$ up to biholomorphisms isotopic to the identity. Let $Q^1(\Sigma) \subset Q(\Sigma)$ be the subspace of unit-area quadratic differentials. We can identify $\Flat$ with a quotient space of $Q^1(\Sigma)$ as follows.

Recall that a quadratic differential $q$ admits an atlas of natural charts $\{(U_i,w_i)\}$, in which $q$ takes the local form $dw_i^2$ away from its zeros and marked points. The group $\SL_2\R$ acts on $Q(\Sigma)$ by applying real-linear transformations to these natural coordinates. Namely, if $q\in Q(\Sigma)$ has a natural atlas $\{(U_i,w_i)\}$ and $A\in\SL_2\R$, then $Aq$ is defined by the atlas $\{(U_i,A\circ w_i)\}$. More generally, for $A\in\mathrm{GL}_2^+(\R)$, the area changes by
\[
    \Area(Aq)=\abs{\det(A)}\Area(q).
\]
Thus the action of $\SL_2\R$ preserves area and restricts to an action on $Q^1(\Sigma)$.

Two quadratic differentials induce the same flat metric if and only if they differ by multiplication by a complex number of unit modulus. More precisely, multiplication of a quadratic differential by $e^{i\theta}$ corresponds, in natural coordinates, to rotation by angle $\theta/2$. Therefore it corresponds to the action of the matrix
\[
    \begin{pmatrix}
        \cos(\theta/2) & -\sin(\theta/2) \\
        \sin(\theta/2) & \cos(\theta/2)
    \end{pmatrix}.
\]
In other words, we have the identification
\[
    \Flat \cong Q^1(\Sigma)/\SO_2\R.
\]
In what follows, by a slight abuse of notation, we will use $q$ to denote both the quadratic differential and its underlying flat metric, making no distinction between the two unless necessary.

\subsection{Geodesics}\label{sec:geodesic}

Let $(\Sigma,q)$ be a half-translation surface equipped with its induced flat metric. In this subsection, we recall the basic local structure of geodesics on $(\Sigma,q)$.

Recall that a \emph{geodesic} in a metric space is a curve that locally minimizes distance. Every regular point in $\Sigma\setminus P$ has a Euclidean neighborhood, and hence a geodesic passing through such a point is locally a straight line segment. At a singular point in $P$, a geodesic passing through the point is locally a concatenation of two straight line segments meeting at the cone point. These two segments determine two angles at the cone point. For the concatenated path to be locally minimizing, both angles must be at least $\pi$; otherwise, one can shorten the path by replacing it with a segment bypassing the cone point, as in Figure \ref{fig:geodesic_angle}.

\begin{figure}[htbp]
    \centering
    \includegraphics[width=0.5\textwidth]{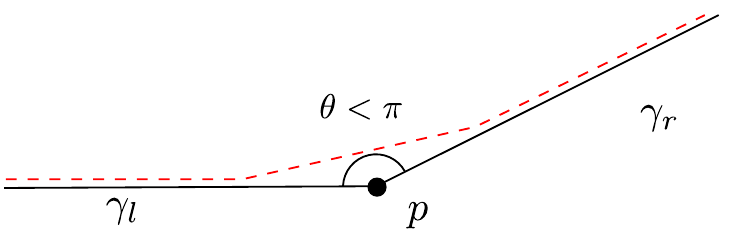}
    \caption{The point $p$ is a cone point, and $\gamma_l$ and $\gamma_r$ are two straight line segments meeting at $p$. If the angle on the left is less than $\pi$, then the red dashed curve is shorter than the concatenation of $\gamma_l$ and $\gamma_r$.}
    \label{fig:geodesic_angle}
\end{figure}

A geodesic connecting two cone points with no other cone points in its interior is called a \emph{saddle connection}. Consequently, any geodesic can be viewed as a concatenation of saddle connections, possibly with ordinary straight line segments at its ends, such that any two consecutive segments make angles of at least $\pi$ at their common cone point.

Among all geodesics, closed geodesics are of particular interest. If a closed geodesic contains no cone points, then the discreteness of the singular set $P$ implies that it has a neighborhood isometric to a Euclidean cylinder. This cylinder can be enlarged until its boundary meets one or more cone points. Such a maximal cylinder is foliated by closed geodesics homotopic to its core curve. Its boundary components are themselves closed geodesics, formed by concatenations of saddle connections, and the angle facing the interior of the cylinder is exactly $\pi$ at each cone point.

We refer to the closed geodesics in the interior of such a cylinder, as well as its boundary curves, as \emph{cylinder curves}. A closed geodesic that is not a cylinder curve is called a \emph{singular closed geodesic}. For such a curve, there must be at least one cone point on each side of the geodesic where the angle subtended by adjacent saddle connections is strictly greater than $\pi$.

\begin{figure}[htbp]
    \centering
    \includegraphics[width=0.5\textwidth]{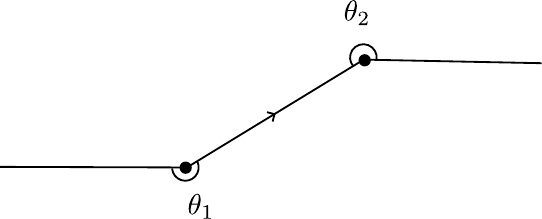}
    \caption{A local picture of a singular closed geodesic, oriented from left to right. The angles $\theta_1$ and $\theta_2$ lie on two different sides of the curve and are both greater than $\pi$.}
\end{figure}

\subsection{CAT(0) properties}

For a flat metric on a half-translation surface, all cone angles are strictly greater than $2\pi$, which implies that the metric is locally $\CAT$ (see \cite[Theorem II.5.5]{bridsonMetricSpacesNonPositive1999}). Note that the complete surface $\widehat{\Sigma}$ may not be locally $\CAT$, because it is allowed to have cone angle $\pi$ at marked points. 

In this subsection, we recall several fundamental properties of $\CAT$ metrics and their applications to flat metrics, with a particular focus on the universal cover. For a detailed introduction to $\CAT$ spaces, we refer the reader to \cite{bridsonMetricSpacesNonPositive1999}. 

By the Cartan-Hadamard theorem, if $S$ is equipped with a complete metric which is locally $\CAT$, then the universal cover $\widetilde{S}$ is a $\CAT$ space homeomorphic to $\R^2$. The Gromov boundary $\partial \widetilde{S}$ is defined to be the set of geodesic rays with a fixed basepoint up to bounded Hausdorff distance. $\partial\til S$ is homeomorphic to the circle $S^1$, and every geodesic in $\widetilde{S}$ has two distinct endpoints in $\partial \widetilde{S}$. We collect several properties of geodesics on surfaces equipped with locally $\CAT$ metrics, as well as properties of their universal covers.

\begin{prop}\label{p.CAT(0)}
    Let $S$ be a surface equipped with a locally $\CAT$ metric. Then the following hold:
    \begin{itemize}
        \item Every free homotopy class of closed curves contains a geodesic representative that minimizes length among all curves in its class \cite[Theorem II.4.13]{bridsonMetricSpacesNonPositive1999}.
        \item Any two points in the universal cover $\til S$ are connected by a unique geodesic \cite[Proposition II.1.4]{bridsonMetricSpacesNonPositive1999}. 
        \item The distance function $d \colon \til S \times \til S \to \R$ is convex \cite[Proposition II.2.2]{bridsonMetricSpacesNonPositive1999}.
        \item Any two distinct points in the Gromov boundary $\partial \til S$ are connected either by a unique geodesic in $\til S$, or by a family of parallel geodesics that foliate a Euclidean strip. In the latter case, each geodesic in this Euclidean strip projects to a closed geodesic on $S$ \cite[Theorems II.2.13, II.9.33]{bridsonMetricSpacesNonPositive1999}.
    \end{itemize}
\end{prop}

In the case of half-translation surfaces, if the set of marked points $\mc E$ is nonempty, then the metric is not complete. To resolve it, we follow \cite[\S 2.4]{duchin2010length} and consider the universal covering $\til\Sigma$ and its metric completion $\bar{\Sigma}$. $\til\Sigma$ can be regarded as branching covering which is branched over $\mc E$. Every marked point has a neighborhood in $\Sigma$ which lifts to a Euclidean cone of infinite angle in $\til\Sigma$. Then $\bar{\Sigma}$ is obtained from $\til\Sigma$ by adding one point to each lift of such a neighborhood, corresponding to the cone point of infinite cone angle. Then $\bar{\Sigma}$ is a complete $\CAT$ metric space. 

Any homotopy class of closed curves $\gamma\subset\Sigma$ admits a geodesic representative, which is the projection of the corresponding geodesic in $\bar{\Sigma}$. We define its length with respect to the flat metric $q$ as the length of its geodesic representative, denoted by $\ell_q(\gamma)$. 

Next, we define the transverse intersection of geodesics for flat metrics. Two geodesics $\alpha$ and $\beta$ in $\Sigma$ are said to \emph{intersect transversely} if there exist lifts $\til\alpha$ and $\til\beta$ with endpoints $\alpha^\pm, \beta^\pm \in \partial \til \Sigma$ such that the pairs $\{\alpha^\pm\}$ and $\{\beta^\pm\}$ interlace. More precisely, this occurs if each connected component of $\partial \til \Sigma \setminus \{\alpha^+, \alpha^-\}$ contains exactly one of the points $\{\beta^+, \beta^-\}$.

\begin{figure}[htbp]
    \centering
    \includegraphics[width=0.5\textwidth]{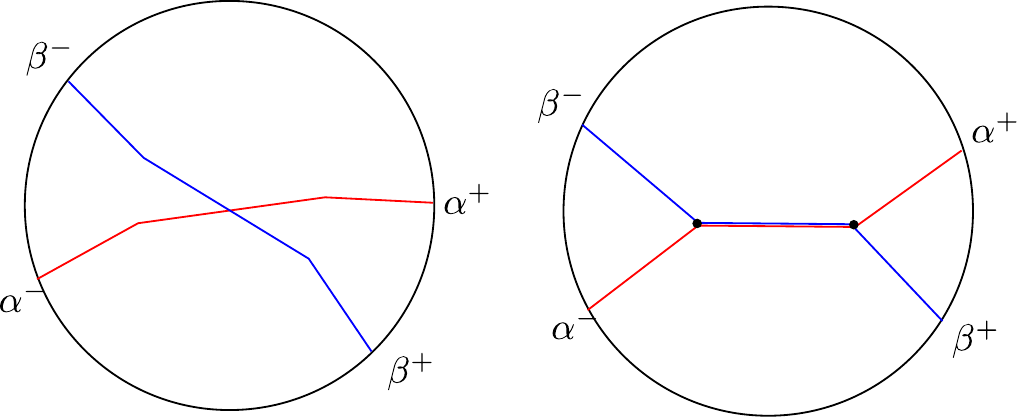}
    \caption{Two different cases of transverse intersection of geodesics $\alpha,\beta$. On the left $\alpha\cap\beta$ is a single point, while on the right $\alpha\cap\beta$ is a concatenation of saddle connections.}
\end{figure}

\begin{prop}[{\cite[Proposition 3.4]{pozzettiFinslerMetrics$12026}}]\label{p.transverse_intersection}
    Let $(\Sigma, q)$ be a surface equipped with a flat metric. If two geodesics $\alpha$ and $\beta$ intersect transversely, then for each connected component of their intersection, one of the following holds:
    \begin{enumerate}
        \item The component is a single point $p$, and the two geodesics cross transversely at $p$.
        \item The component is a concatenation of saddle connections; in this case, $\alpha$ approaches from one side of $\beta$, coincides with $\beta$ along these common saddle connections, and then exits on the opposite side.
    \end{enumerate}
    Furthermore, if either $\alpha$ or $\beta$ is a cylinder curve, the intersection consists only of isolated points (Case 1).
\end{prop}

\subsection{Measured foliations}
Measured foliations play an important role in Teichm\"uller theory and the geometry of surfaces. They are also closely related to half-translation surfaces, as a half-translation surface gives rise to a family of measured foliations parametrized by direction. In this subsection, we recall the formal definition of measured foliations and clarify their relationship with half-translation surfaces.

\begin{defn}
Let $\Sigma=\widehat{\Sigma}\setminus \mc E$ be a finite-type topological surface, where $\mc E$ is the set of puncture points. A \emph{singular foliation} $\mathcal{F}$ on $\Sigma$ consists of a finite subset $P$ of singular points and a covering of $\Sigma\setminus P$ with charts $\phi_i:U_i\to \R^2$ such that the transition functions on overlaps are of the form  \[
    f_{ij}(x, y) = (h_{ij}(x, y), g_{ij}(y))
    \]
    where $h_{ij}$ is continuous and $g_{ij}$ is a strictly monotonic continuous function. Each point in $P$ is $p$-pronged for some $p\ge 3$, while points in $\mc E$ are allowed to be 1-pronged singularities.

% is specified by the following data:
% \begin{enumerate}
%     \item Singular set: A finite subset $P = \{p_1, \dots, p_n\} \subset \Sigma$, called the set of singular points.
    
%     \item Regular atlas: An open cover $\{U_i\}_{i \in I}$ of the punctured surface $\Sigma \setminus P$ equipped with homeomorphisms $\phi_i: U_i \to \mathbb{R}^2$. For any non-empty intersection $U_i \cap U_j$, the transition map 
%     \[
%     f_{ij} = \phi_j \circ \phi_i^{-1} : \phi_i(U_i \cap U_j) \to \phi_j(U_i \cap U_j)
%     \]
%     takes the form
%     \[
%     f_{ij}(x, y) = (h_{ij}(x, y), g_{ij}(y))
%     \]
%     where $h_{ij}$ is continuous and $g_{ij}$ is a strictly monotonic continuous function.
    
%     \item Singular local models: For each singular point $p \in P$, there exists an open neighborhood $V$ of $p$ such that $V \cap P = \{p\}$, an integer $k \geq 1$, and a homeomorphism 
%     \[
%     \psi: V \to D \subset \mathbb{C}
%     \]
%     where $D$ is an open disk centered at the origin, and $\psi(p) = 0$. In the punctured disk $D \setminus \{0\}$, the topological foliation induced by the regular charts coincides with the $(k+2)$--pronged foliation
%     \[
%     \operatorname{Im}(z^{k/2} dz) = 0.
%     \] 
%     \item Local model of punctures: the same as singular local models except $k\ge -1$.
% \end{enumerate}
\end{defn}

\begin{defn}\label{def:measured_foliation}
    Let $\mc F$ be a singular foliation on $\Sigma$. A \emph{transverse invariant measure} $\mu$ is an assignment of a measure to each transverse arc such that if two transverse arcs $\alpha$ and $\beta$ are isotopic through transverse arcs whose endpoints remain in the same leaf, then $\mu(\alpha)=\mu(\beta)$. A \emph{measured foliation} is a singular foliation with a transverse invariant measure.
\end{defn}

The space of measured foliations on $\Sigma$ up to isotopy and Whitehead equivalence, denoted by $\mc{MF}(\Sigma)$, is homeomorphic to $\mb R^{6g-6+2n}$. Its projectivization with respect to positive scalar multiplication on transverse measures is denoted by $\mc{PMF}(\Sigma)$.

There is a structure theorem for measured foliations about the decomposition into minimal components and annuli. We state here the version in \cite{vianaDynamicsIntervalExchange} which deals with measured foliations on a closed surface. For a surface with punctures, we can pass to the orienting double cover which branches over odd-pronged singularities. The classification upstairs descends to the decomposition of the original measured foliation.

\begin{thm}[{\cite[Theorem 3.13]{vianaDynamicsIntervalExchange}}]\label{thm:foliation_decomposition}
    For a measured foliation on the surface, there exists finitely many pairwise disjoint open domains $D_i$ whose closures cover the surface, such that for each $i$
    \begin{itemize}
        \item either $D_i$ is homeomorphic to a cylinder and foliated by closed leaves,
        \item or $D_i$ consists of non-closed leaves, each of which is dense in $D_i$,
    \end{itemize}
    and in either case, $D_i$ are sub-surfaces with boundary consisting of saddle connections.
\end{thm}

Given a half-translation surface represented by a quadratic differential $q$, we can define a family of measured foliations $\{\mc F_\theta\}_{\theta\in\R/\pi\Z}$ as follows. Following Definition \ref{def:measured_foliation}, we can give the foliation structure by natural charts in Lemma \ref{l:standard_chart} composed with $\theta$-rotation with the Lebesgue measure. The leaves are locally straight lines in natural charts which make angle $\theta$ with the horizontal direction. The family of such measured foliations are called \emph{directional foliations}. We denote the directional foliation with parameter $\theta$ by $\mc F_\theta^q$ or simply $\mc F_\theta$. In particular, we call $\mc F_0$ the \emph{horizontal foliation} and $\mc F_{\pi/2}$ the \emph{vertical foliation}.

\subsection{Geodesic currents on closed surfaces}\label{sec:geodesic_current}

In this subsection, we review the theory of geodesic currents on closed surfaces developed by Bonahon \cite{bonahon1988geometry}. Geodesic currents generalize both closed curves and measured foliations, providing a powerful tool for the proof of our main results.

Given a closed surface $\Sigma$ equipped with a hyperbolic metric $h$, consider the space
\[
G(\til \Sigma) := \big( (\partial\til \Sigma \times \partial\til \Sigma) \setminus \Delta \big) / {\sim},
\]
where $\Delta$ is the diagonal and the equivalence relation is given by $(x,y) \sim (y,x)$. This space is naturally identified with the space of unoriented bi-infinite geodesics in $\til \Sigma$ up to bounded Hausdorff distance. Because different metrics induce homeomorphic Gromov boundaries, we typically omit the ambient metric $h$.

\begin{defn}
    A \emph{geodesic current} is a $\pi_1(\Sigma)$-invariant, locally finite Borel measure on $G(\til \Sigma)$.
\end{defn}

Let $\mathcal{C}(\Sigma)$ denote the space of geodesic currents on $\Sigma$. A fundamental class of examples arises from closed curves. For any closed curve $\gamma$ in $\Sigma$, consider its geodesic representative. Its preimage in the universal cover $\til \Sigma$ consists of a countable family of bi-infinite geodesics. The endpoints of these geodesics form a discrete subset of $G(\til \Sigma)$. By assigning a Dirac mass to each point in this subset, we construct a counting measur, which is a geodesic current. By an abuse of notation we also denote it by $\gamma$.

There is also a natural embedding of $\mc{MF}(\Sigma)$ into $\Curr(\Sigma)$ (see \cite[\S 11.3]{kapovich2001hyperbolic}). This embedding is compatible with the decomposition of measured foliations described in Theorem \ref{thm:foliation_decomposition}. Specifically, if a measured foliation $\mc F$ decomposes into open domains $D_i$ as in Theorem \ref{thm:foliation_decomposition}, the foliation restricted to each $D_i$ can be extended to a measured foliation on the entire surface. This is achieved by choosing a spine for the complement of $D_i$ and foliating it with leaves parallel to the boundary of $D_i$ (for details, see \cite[\S 5.4]{fathiThurstonsWorkSurfaces2012}). By a slight abuse of notation, we use $\mc F$ to denote both the measured foliation and its associated geodesic current, and we let $\mc F_i$ denote the geodesic current associated with $D_i$. In the space of geodesic currents, we then have the decomposition
\[ \mc F = \sum_i \mc F_i. \]

Furthermore, Bonahon \cite{bonahon1988geometry} constructed a continuous, symmetric bilinear intersection form on $\Curr(\Sigma) \times \Curr(\Sigma)$ that extends the geometric intersection number of closed curves. Let $\mc P \subset G(\til \Sigma) \times G(\til \Sigma)$ be the subspace consisting of pairs of linked endpoints. Two geodesic currents $\alpha$ and $\beta$ induce a $\pi_1(\Sigma)$-invariant product measure $\alpha \times \beta$ on $G(\til \Sigma) \times G(\til \Sigma)$. The restriction of this measure to $\mc P$ descends to a measure on the quotient space $\mc P/\pi_1(\Sigma)$, which we still denote by $\alpha \times \beta$.

\begin{thm}[\cite{bonahon1988geometry}]\label{intersection number}
    There exists a continuous, bilinear and symmetric function\[
    i: \Curr(\Sigma) \times \Curr(\Sigma) \to \R
    \]
    defined by
    \[
    i(\alpha,\beta)=\int_{\mc P/\pi_1(S)}\alpha\times\beta\] for geodesic currents $\alpha,\beta\in\Curr(\Sigma)$. It restriction to closed curves agrees with the geometric intersection number.
\end{thm}

% Apart from closed curves and measured foliations, there is another important class of geodesic currents called \emph{Liouville currents}, which record the length information of closed curves for a given metric.

% \begin{defn}
%     Given a metric $g$ on $\Sigma$, a geodesic current is called the \emph{Liouville current} of $g$, denoted by $L_g$, if 
%     \[ i(L_g, \gamma) = \ell_g(\gamma) \]
%     for every closed curve $\gamma \in \pi_1(\Sigma)$, where $\ell_g(\gamma)$ is the minimum length among all curves in the free homotopy class of $\gamma$.
% \end{defn}

% A geodesic current is uniquely determined by its intersection numbers with all closed curves \cite[Theorem 2]{otalSpectreMarqueLongueurs1990}. Consequently, the Liouville current of a metric is unique whenever it exists.

% \begin{rmk}
%     Bankovich and Leininger defined in \cite[\S 3.2]{bankovic2018marked} the geodesic current for more general flat cone metrics. They first defined a measure on $T^1(\Sigma\setminus P)$ called pre-Liouville current and then push it forward to $G(\widehat{\Sigma})$.
% \end{rmk}

\subsection{Geodesic currents on surfaces with punctures}

When $\Sigma$ is a punctured surface, we take the viewpoint of \cite{burgerCurrentsSystolesCompactifications2021} and include the relevant results in this subsection. Let $\Sigma$ be equipped with a complete hyperbolic metric $h$ of finite area. Then similar to the case of closed surfaces, a \emph{geodesic current} is a $\pi_1(\Sigma)$-invariant, locally finite Borel measure on $G(\wideS)$.

The subtlety is about the intersection number. On surfaces with punctures, the intersection number is not always continuous \cite{sasaki2022currents}. To resolve this, we consider a subspace of geodesic currents defined in the following way. Given a geodesic current $\mu\in\Curr(\Sigma)$, let the \emph{carrier} $\Carr(\mu)\subset \Sigma$ be the closed subset obtained by projecting all geodesics to $\Sigma$ whose endpoints are in the support of $\mu$. For a compact subset $K\subset\Sigma$, define \[
\Curr_K(\Sigma)=\set{\mu\in\Curr(\Sigma)}{\Carr(\mu)\subset K}.\]

\begin{thm}[{\cite[Theorem 2.4]{burgerCurrentsSystolesCompactifications2021}}]
    For every compact subset $K\subset \Sigma$, the intersection \[
    i:\Curr(\Sigma)\times \Curr_K(\Sigma)\to\R\]
    is continuous and finite.
\end{thm}

\begin{rmk}
    The finiteness is not contained in \cite[Theorem 2.4]{burgerCurrentsSystolesCompactifications2021}, but in a paragraph before the theorem. It is also easy to verify because of the compactness of the carrier.
\end{rmk}

It is a classical fact that $\mc{MF}(\Sigma)$ is homeomorphic to the space of measured laminations with compact support, denoted by $\mc{ML}_c(\Sigma)$. One can prove it, for example, by showing that both $\mc{MF}(\Sigma)$ and $\mc{ML}_c(\Sigma)$ are closure of weighted simple closed curves, see \cite[Proposition 2.11]{kahnConformalSurfaceEmbeddings2022} and \cite[Theorem 3.1.3]{pennerCombinatoricsTrainTracks1992}. Moreover, similar to the case of closed surfaces, there is a natural embedding of $\mc{MF}(\Sigma)$ into the space of geodesic currents $\Curr(\Sigma)$.

\begin{lemma}
    There exists a compact subset $K\subset\Sigma$ such that $\mc{MF}(\Sigma)\hookrightarrow\Curr_K(\Sigma)$.
\end{lemma}

\begin{proof}
    We view measured foliations as measured laminations with compact support. By \cite[Lemma 2.12]{erlandssonMirzakhanisCurveCounting2022}, there is a compact subset $K\subset\Sigma$ such that every measured lamination with compact support is contained in $K$. 
\end{proof}

Now we focus on the flat metrics on $\Sigma$. Duchin--Leininger--Rafi \cite{duchin2010length} constructed a geodesic current which records the length information of closed curves for a given flat metric. 

\begin{thm}[{\cite[Lemma 9, Proposition 26]{duchin2010length}}]\label{thm:liouville_current}
    Let $q$ be a flat metric of unit area on $\Sigma$, represented by a quadratic differential, and let $\mc F_\theta$ be the measured foliation associated with $q$ in the direction $\theta$. Then \[
    L_q=\frac12 \int_0^\pi\mc F_\theta\,d\theta\]
    is a well-defined geodesic current. Moreover, 
    \begin{enumerate}
        \item for any closed curve $\gamma\subset\Sigma$, we have $i(L_q,\gamma)=\ell_q(\gamma)$.
        \item $i(L_q, L_q)=\frac{\pi}{2} $.\label{p:self_intersection}
    \end{enumerate}
    % Its Liouville current is given by 
    % \[ L_q = \frac{1}{2} \int_0^\pi \mc F_\theta \, d\theta, \]
    % where $\mc F_\theta$ is the measured foliation associated with $q$ in the direction $\theta$. Here, the integral is defined as a limit of Riemann sums in the space of geodesic currents.
\end{thm}

Since all measured foliations correspond to geodesic currents in $\Curr_K(\Sigma)$ for some compact subset $K\subset\Sigma$, as a linear combination of measured foliations, $L_q$ is also a geodesic current in $\Curr_K(\Sigma)$. Then the intersection $i(L_q,\cdot)$ is well-defined and finite. We call it the \emph{Liouville current} of the flat metric $q$. 

% As is the case for hyperbolic and negatively curved metrics, the self-intersection of the Liouville current of a flat metric recovers the total area of the surface.

\section{Basic examples}\label{examples}

In this section we show that, in certain special cases, one can compute explicitly the value of $K(q_1,q_2)$, where $K$ is the function defined in Equation~\eqref{equ:K} and $q_1,q_2$ are flat metrics.

\subsection{Affinely equivalent metrics}\label{sec:ex}

Recall the action of $\SL_2\R$ on the space of quadratic differentials, which acts $\R$-linearly on the natural coordinates. We say two flat metrics are \emph{affinely equivalent} if they admit quadratic differential representatives lying in the same $\SL_2\R$-orbit. By standard linear algebra, the singular values of any matrix $A \in \SL_2\R$ are $\lambda_A$ and $\lambda_A^{-1}$, where we conventionally assume $\lambda_A \ge 1$.

\begin{prop}\label{p:baby_case}
    Let $q_1, q_2 \in \Flat$ be two affinely equivalent flat metrics. If their quadratic differential representatives (still denoted $q_1$ and $q_2$) satisfy $q_2 = A q_1$ for some matrix $A \in \SL_2\R$, then 
    \[ K(q_1, q_2) = K(q_2, q_1) = \log \lambda_A. \]
\end{prop}

Note that $\lambda_A = 1$ if and only if $A \in \SO_2\R$. In particular, this implies $K(q_1, q_2) \ge 0$, with equality if and only if $q_1$ and $q_2$ induce the same flat metric. This establishes a special case of Corollary \ref{cor:nondegenerate}.

\begin{proof}
    Let $A = U \Lambda V$ be the singular value decomposition of $A$, where $U, V \in \SO_2\R$ and $\Lambda = \operatorname{diag}(\lambda_A, \lambda_A^{-1})$. Because flat metrics are invariant under rotation, $V q_1$ and $U^{-1} q_2$ represent the same flat metrics in $\Flat$ as $q_1$ and $q_2$, respectively. Furthermore, they satisfy $U^{-1} q_2 = \Lambda (V q_1)$. Thus, without loss of generality, we may assume that the matrix $A$ connecting the representatives $q_1$ and $q_2$ is the diagonal matrix $\Lambda$.

    Now, consider the action of $\Lambda$ on the natural charts. Let $s$ be a saddle connection of the metric $q_1$ with length $l$ and direction $\theta$. Its horizontal and vertical projections have lengths $l\cos\theta$ and $l\sin\theta$, respectively. The matrix $\Lambda$ maps $s$ to a saddle connection $s'$ in the metric $q_2$, whose horizontal and vertical projections have lengths $\lambda_A l \cos\theta$ and $\lambda_A^{-1} l \sin\theta$. Thus, the length of the mapped saddle connection $s'$ is 
    \begin{equation}\label{equ:length_saddle}
        l' = \sqrt{(\lambda_A l \cos\theta)^2 + (\lambda_A^{-1} l \sin\theta)^2}.
    \end{equation}
    It follows that $l' \le \lambda_A l$, with equality if and only if $\theta = 0$.

    Any closed geodesic $\gamma$ in the metric $q_1$ is a concatenation of finitely many saddle connections, each of which is mapped to a saddle connection in $q_2$. The concatenation of these image segments forms a curve in the same free homotopy class as $\gamma$. However, this image curve may not be a true geodesic in $q_2$, as the affine map alters angles and may violate the $\pi$-angle condition described in Section \ref{sec:geodesic}. Therefore, we have
    \[ \ell_{q_2}(\gamma) \le \text{length of the image curve} \le \lambda_A \ell_{q_1}(\gamma). \]
    Taking the supremum over all closed curves $\gamma$ yields the upper bound
    \[ K(q_1, q_2) \le \log\lambda_A. \]

    To establish the lower bound, we apply a result of Masur \cite[Theorem 2]{masurClosedTrajectoriesQuadratic1986}, which states that the set of directions admitting cylinder curves is dense. We can therefore choose a sequence of cylinder curves $\gamma_n$ in directions $\theta_n$ such that $\theta_n \to 0$. Because a cylinder curve in $q_1$ consists of parallel saddle connections, it maps to a cylinder curve in $q_2$, which is necessarily a geodesic. By Equation \eqref{equ:length_saddle}, we obtain
    \[ \lim_{n \to \infty} \log \frac{\ell_{q_2}(\gamma_n)}{\ell_{q_1}(\gamma_n)} = \log\lambda_A, \]
    which provides the lower bound. Combining the two bounds yields $K(q_1, q_2) = \log\lambda_A$. 
    
    Finally, since the singular values of $A^{-1}$ are identically $\lambda_A$ and $\lambda_A^{-1}$, applying the same argument in reverse gives $K(q_2, q_1) = \log\lambda_A$.
\end{proof}

\subsection{Genus 2 examples}

We now present an example demonstrating that $K(\cdot,\cdot)$ is not always symmetric. In fact, asymmetry is expected to be the generic behavior (see Section \ref{topology}). 

Consider two squares, each of area $1/2$, as shown in Figure \ref{genus 2}. By identifying parallel sides with matching labels, we obtain a flat surface of genus 2. Varying the length of the saddle connection labeled IV produces a family of translation surfaces. Let $q_s$ denote the translation structure in which the length of saddle connection IV is $s \in (0, 1/\sqrt{2})$.

\begin{figure}[htbp]
    \centering
    \includegraphics[width=0.7\textwidth]{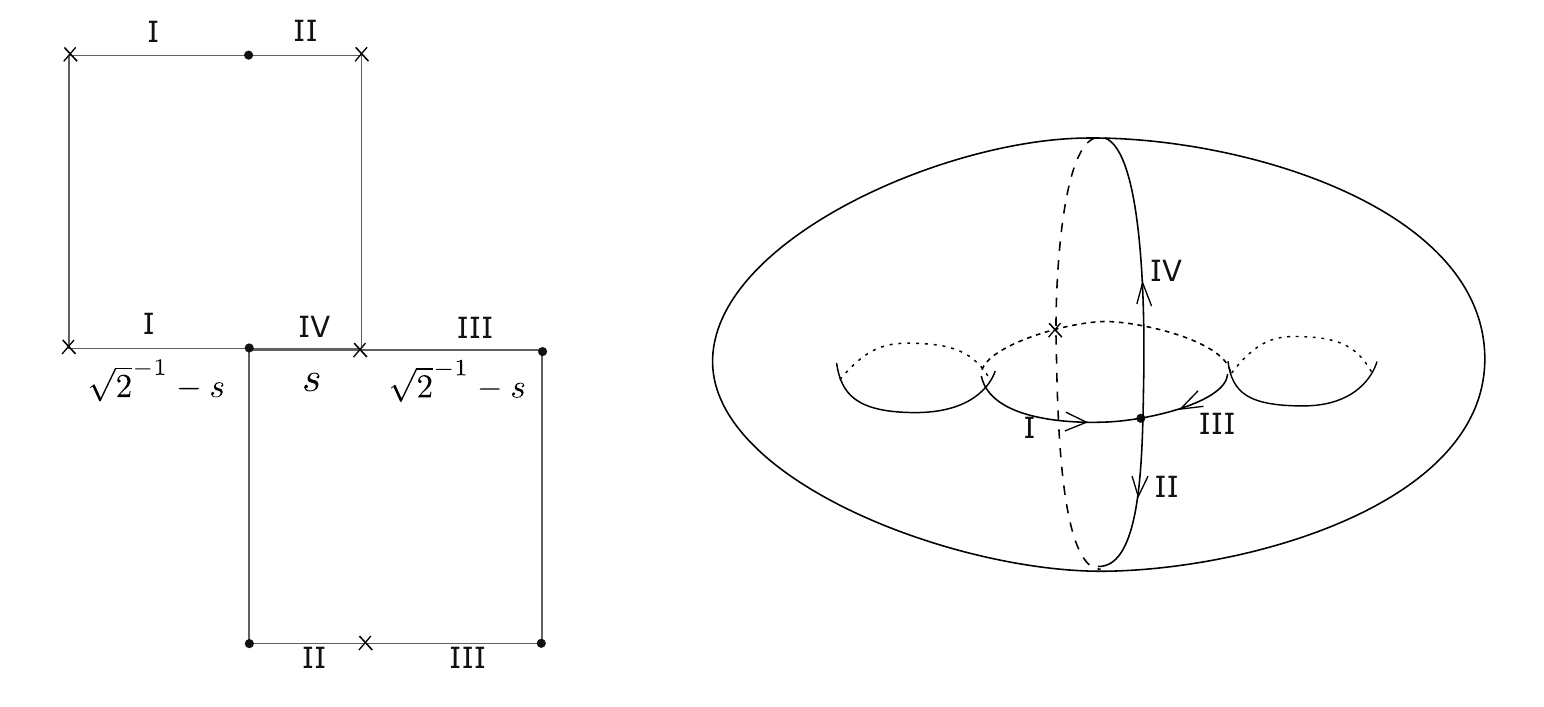}
    \caption{The surface on the right is obtained by identifying the parallel sides of the polygons on the left. All labeled sides in the left-hand figure are oriented from left to right, which induces an orientation on the corresponding segments illustrated on the right. The vertices of the polygons are glued to form two cone points.}
    \label{genus 2}
\end{figure}

Fix $a, b \in \R_+$ such that $0 < a < b < 1/\sqrt{2}$. There exists a Lipschitz map $f \colon q_a \to q_b$ isotopic to the identity that fixes the cone points and maps vertical segments to vertical segments isometrically. Furthermore, $f$ linearly compresses saddle connections I and III by a factor of $\frac{1/\sqrt{2} - b}{1/\sqrt{2} - a}$, and linearly stretches saddle connections II and IV by a factor of $\frac{b}{a}$. Consequently, the Lipschitz constant of $f$ is
\[
\operatorname{Lip}(f) := \sup_{x \neq y} \frac{d(f(x), f(y))}{d(x, y)} = \frac{b}{a}.
\]
Any geodesic in $q_a$ is mapped to a curve in $q_b$ whose length is at most $\operatorname{Lip}(f)$ times the length of the original geodesic. Therefore, we have $K(q_a, q_b) \le \log \operatorname{Lip}(f) = \log\frac{b}{a}$. On the other hand, the concatenation $\mathrm{II} * \mathrm{IV}^{-1}$ forms a simple closed geodesic whose length is stretched by exactly $\frac{b}{a}$, which yields the exact value
\[ K(q_a, q_b) = \log\frac{b}{a}. \]

Similarly, the inverse map $f^{-1} \colon q_b \to q_a$ is a Lipschitz map isotopic to the identity with $\operatorname{Lip}(f^{-1}) = \frac{1/\sqrt{2} - a}{1/\sqrt{2} - b}$. The closed curve $\mathrm{I} * \mathrm{III}^{-1}$ is a simple closed geodesic whose length is stretched by exactly this factor. It follows that
\[ K(q_b, q_a) = \log\frac{1/\sqrt{2} - a}{1/\sqrt{2} - b}. \]
It is clear that for a generic choice of $a$ and $b$, we have $K(q_a, q_b) \neq K(q_b, q_a)$.

\subsection{Closed curves versus simple closed curves}

We now present an example of two flat metrics on a pair of pants to demonstrate that, in the definition of $K(\cdot,\cdot)$, taking the supremum over simple closed curves versus all closed curves can yield different quantities. In this example, all boundary curves are cylinder curves, which ensures we do not need to worry about potential singularities on the boundaries. A special feature of the pair of pants is that all simple closed curves are freely homotopic to multiples of the boundary curves.

\begin{figure}[htbp]
    \centering
    \includegraphics[width=0.6\textwidth]{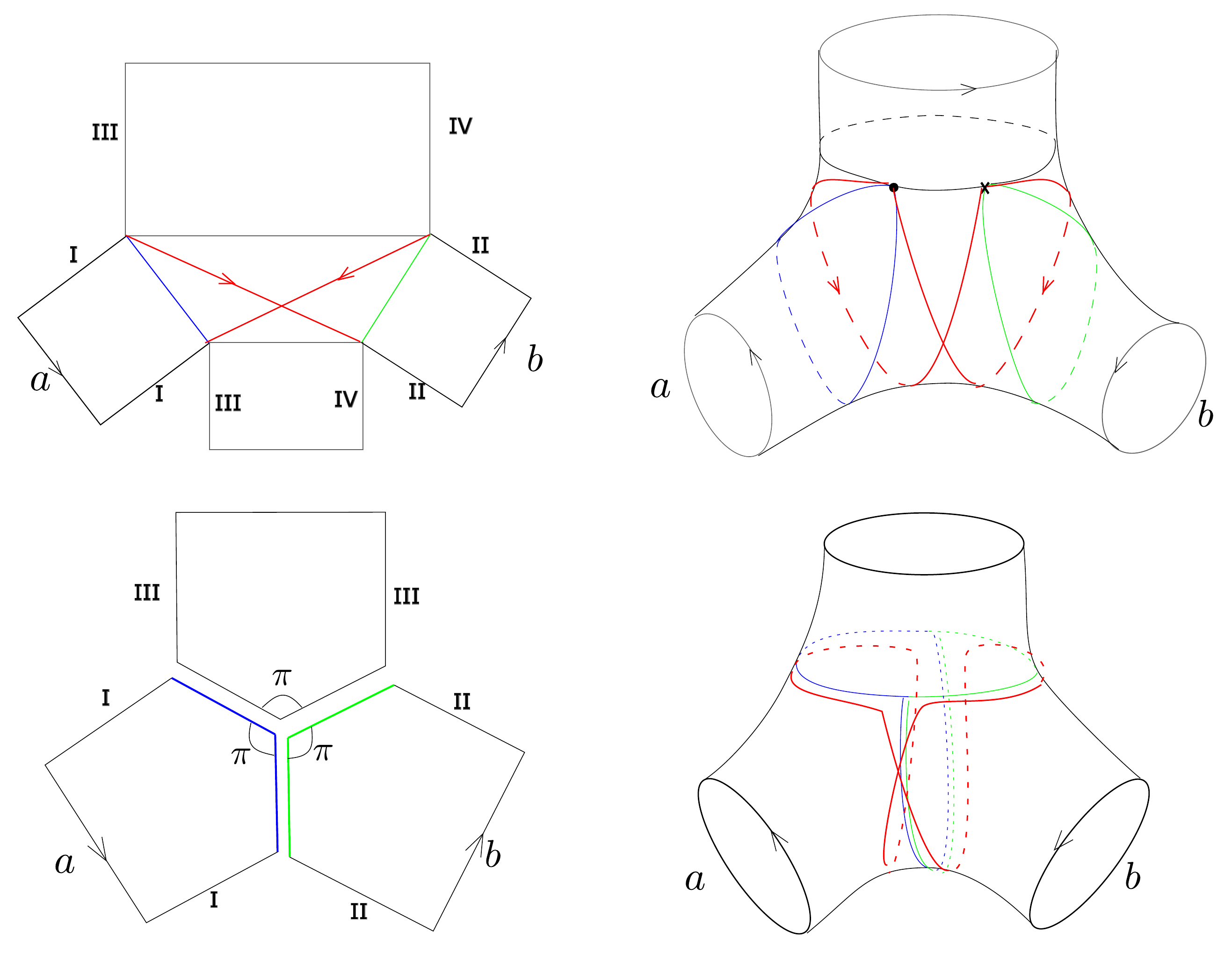}
    \caption{A pair of pants with cylinder boundary curves $a, b, c$. Pairs of sides with the same label in the polygon are glued together; sides I and II are glued preserving orientation, while sides III and IV are glued with opposite orientation. The red curve is the geodesic representative of $a*b\inv$.}
    \label{fig:pair_of_pants}
\end{figure}

Consider the upper pair of pants equipped with the metric $q_1$ in Figure \ref{fig:pair_of_pants}. The only primitive simple closed curves are the three boundary curves $a, b,$ and $c$. The concatenation $a*b\inv$ is a closed curve with self-intersections. From the geometry of the figure, we observe that $\ell_{q_1}(a*b\inv) > \ell_{q_1}(a) + \ell_{q_1}(b)$. Conversely, for the lower pair of pants with metric $q_2$, we have $\ell_{q_2}(a*b\inv) = \ell_{q_2}(a) + \ell_{q_2}(b)$. Let us further assume that the ratio $\frac{\ell_{q_1}(\gamma)}{\ell_{q_2}(\gamma)}$ is the same for any boundary curve $\gamma$. It then follows that
\[
\frac{\ell_{q_1}(a*b\inv)}{\ell_{q_2}(a*b\inv)} > \frac{\ell_{q_1}(a)+\ell_{q_1}(b)}{\ell_{q_2}(a)+\ell_{q_2}(b)} = \sup_{\gamma \text{ simple}} \frac{\ell_{q_1}(\gamma)}{\ell_{q_2}(\gamma)},
\]
where the supremum on the right is taken over all simple closed curves.

Note that for closed surfaces, it remains an open question whether these two suprema are equal.

\section{Proof of the main theorem}\label{proof}

This section is devoted to the proof of Theorem \ref{main theorem}. The proof proceeds in three steps, each presented in its own subsection.

\subsection{Equality of cylinder curve lengths}

In this subsection, we use geodesic currents and intersection numbers to show that the equality of the areas of $q_1$ and $q_2$ forces the equality of cylinder curve lengths. In fact, we establish a stronger result: for any minimal component of a measured foliation induced by either $q_1$ or $q_2$, its intersection numbers with the Liouville currents of $q_1$ and $q_2$ coincide. Unless otherwise stated, all results in this subsection are proved under the assumptions of Theorem \ref{main theorem}.

\begin{lemma}\label{l:foliation_inequality}
    If $\mc F$ is a measured foliation on $\Sigma$, then $i(L_{q_1},\mc F) \le i(L_{q_2},\mc F)$.
\end{lemma}

\begin{proof}
    This statement follows from the density of weighted simple closed curves in the space of measured foliations, combined with the continuity of the intersection form.
\end{proof}

Let $\mc F^{q_i}_{\theta}$ denote the measured foliation induced by $q_i$ in the direction $\theta$, for $i=1,2$. The following two results hold independently of the hypotheses of Theorem \ref{main theorem}.

\begin{thm}[{\cite[Theorem 12]{duchin2010length}}]\label{thm:topology_current}
    Let $\mu_k$ be a sequence of geodesic currents. Then $\mu_k \to \mu$ if and only if $i(\mu_k,\gamma) \to i(\mu,\gamma)$ for every closed curve $\gamma$.
\end{thm}

\begin{cor}\label{cor:continuous_mf}
    For each fixed $i$, the family of geodesic currents $\{ \mc F^{q_i}_{\theta} \mid \theta \in [0,2\pi) \}$ is continuous with respect to $\theta$.
\end{cor}
\begin{proof}
    Any closed geodesic $\gamma$ can be decomposed as a concatenation $\gamma = \gamma_1 * \dots * \gamma_n$, where each $\gamma_j$ is a saddle connection with direction $\theta_j$. Note that 
    \[
    i(\mc F_\theta^{q_i}, \gamma) = \sum_j \ell(\gamma_j)\abs{\sin(\theta-\theta_j)}.
    \]
    Applying Theorem \ref{thm:topology_current} yields the desired continuity.
\end{proof}

Building upon Lemma \ref{l:foliation_inequality} and the relationship between Liouville currents and area, we obtain the following equality.

\begin{lemma}\label{l:foliation_equality}
    For $i=1,2$ and every direction $\theta$, we have: 
    \[ i(L_{q_1}, \mc F_\theta^{q_i}) = i(L_{q_2}, \mc F_\theta^{q_i}). \]
\end{lemma}
\begin{proof}
    By Proposition \ref{p:self_intersection}, the assumption that $q_1$ and $q_2$ have the same area implies $i(L_{q_1},L_{q_1}) = i(L_{q_2},L_{q_2})$. 
    
    Applying Theorem \ref{thm:liouville_current} and Lemma \ref{l:foliation_inequality}, we obtain the following inequality: 
    \begin{equation}\label{e.first_intersection}
        i(L_{q_2},L_{q_2}) = \frac{1}{2}\int_0^\pi i(L_{q_2}, \mc F_\theta^{q_2}) \, d\theta \ge \frac{1}{2}\int_0^\pi i(L_{q_1}, \mc F_\theta^{q_2}) \, d\theta.
    \end{equation}
    By expanding $L_{q_2}$ as the integral of $\mc F_\theta^{q_2}$, the rightmost term is exactly $i(L_{q_1},L_{q_2})$. Similarly, expanding $L_{q_1}$ as the integral of $\mc F_\theta^{q_1}$ provides an alternative expression for $i(L_{q_1},L_{q_2})$:
    \[
    \frac{1}{2}\int_0^\pi i(L_{q_1}, \mc F_\theta^{q_2}) \, d\theta = i(L_{q_1},L_{q_2}) = \frac{1}{2}\int_0^\pi i(\mc F_\theta^{q_1},L_{q_2}) \, d\theta.
    \]
    Applying Lemma \ref{l:foliation_inequality} again yields the inequality:
    \begin{equation}\label{e.second_intersection}
        \frac{1}{2}\int_0^\pi i(\mc F_\theta^{q_1},L_{q_2}) \, d\theta \ge \frac{1}{2}\int_0^\pi i(\mc F_\theta^{q_1},L_{q_1}) \, d\theta = i(L_{q_1},L_{q_1}).
    \end{equation}
    Combining all the inequalities above, we find $i(L_{q_2},L_{q_2}) \ge i(L_{q_1},L_{q_1})$. However, since $i(L_{q_2},L_{q_2}) = i(L_{q_1},L_{q_1})$, all inequalities in \eqref{e.first_intersection} and \eqref{e.second_intersection} must in fact be equalities. This implies that $i(L_{q_1},\mc F_\theta^{q_i}) = i(L_{q_2},\mc F_\theta^{q_i})$ for almost every $\theta$ and for $i=1,2$. Because the geodesic current $\mc F_\theta^{q_i}$ varies continuously with $\theta$ (Corollary \ref{cor:continuous_mf}), the equality holds for every $\theta$.
\end{proof}

Recall from Section \ref{sec:geodesic_current} that components $C_j$ of a measured foliation can themselves be treated as measured foliations, and thus as geodesic currents.

\begin{cor}\label{cor:minimal_component}
    If the measured foliation decomposes into components as $\mc F_\theta^{q_i} = C_1 \cup \dots \cup C_n$, then $i(L_{q_1},C_j) = i(L_{q_2},C_j)$ for all $j$.
\end{cor}

\begin{proof}
    Lemma \ref{l:foliation_inequality} implies that $i(L_{q_1},C_j) \le i(L_{q_2},C_j)$ for all $j$. In the space of geodesic currents, we can write $\mc F_\theta^{q_i} = \sum_{j=1}^n C_j$, which yields $i(L_{q_k}, \mc F_\theta^{q_i}) = \sum_j i(L_{q_k}, C_j)$ for $k=1,2$. Therefore, 
    \begin{equation}\label{e.component_length}
        i(L_{q_2},\mc F_\theta^{q_i}) = \sum_j i(L_{q_2},C_j) \ge \sum_j i(L_{q_1},C_j) = i(L_{q_1},\mc F_\theta^{q_i}).
    \end{equation}
    Lemma \ref{l:foliation_equality} forces all inequalities in \eqref{e.component_length} to be strict equalities.
\end{proof}

\begin{defn}
    Given a topological surface $\Sigma$ and a flat metric $q$, the set of \emph{cylinder curves} with respect to $q$, denoted by $\cyl(q)$, is the set of simple closed curves on $\Sigma$ whose geodesic representatives are cylinder curves in $q$.
\end{defn}

Note that every cylinder curve in a flat metric $q$ has a well-defined direction, and thus belongs to a component of the foliation induced by $q$ in that direction. Viewing these as special cases of foliation components, we obtain the following result.

\begin{cor}\label{c:cylinder_same_length}
    For any $\gamma \in \cyl(q_1) \cup \cyl(q_2)$, we have $\ell_{q_1}(\gamma) = \ell_{q_2}(\gamma)$.
\end{cor}
\begin{proof}
    Without loss of generality, assume $\gamma$ is a cylinder curve in $q_1$. Then it is contained in a component of $\mc F_\theta^{q_1}$, which, by Theorem \ref{thm:foliation_decomposition}, is foliated by closed curves parallel to $\gamma$. Thus, this component can be identified with a weighted simple closed curve $c\gamma$, where $c>0$ is the Euclidean height of the cylinder. 

    By Corollary \ref{cor:minimal_component}, we have $i(L_{q_1}, c\gamma) = i(L_{q_2}, c\gamma)$. Because the intersection number $i(\cdot, \cdot)$ is bilinear, this gives $c \, i(L_{q_1},\gamma) = c \, i(L_{q_2},\gamma)$, which simplifies to $\ell_{q_1}(\gamma) = \ell_{q_2}(\gamma)$.
\end{proof}

\subsection{Comparison of cylinder sets}\label{sec: Dehn twist}

In general, there is no direct relationship between the sets of cylinder curves of two distinct flat metrics. For instance, a curve that is realized as a cylinder curve in one metric may pass through cone points in the other. The primary goal of this subsection is to prove the following result:

\begin{prop}\label{p:technical}
    Under the hypotheses of Theorem \ref{main theorem}, we have $\cyl(q_2) \subset \cyl(q_1)$.
\end{prop}

To relate the inclusion $\cyl(q_2) \subset \cyl(q_1)$ to curve length inequalities, we require the following observation regarding the effect of Dehn twists on curve lengths under flat metrics.

\begin{prop}[{\cite[Proposition 18, Lemma 19]{duchin2010length}}]\label{p:length_Dehn}
    Let $q$ be a quadratic differential on the surface $\Sigma$, and let $\alpha$ be a simple closed curve. Then $\alpha \notin \cyl(q)$ if and only if there exists a simple closed curve $\beta$ intersecting $\alpha$ transversely such that 
    \[ \ell_q(T_\alpha\beta) = \ell_q(\beta) + i(\alpha,\beta)\ell_q(\alpha). \]
    Moreover, if $\alpha \in \cyl(q)$, then for any transversely intersecting curve $\beta$, 
    \[ \ell_q(T_\alpha\beta) < \ell_q(\beta) + i(\alpha,\beta)\ell_q(\alpha). \]
\end{prop}

We now state a technical lemma, deferring its proof to the end of this subsection.

\begin{lemma}\label{l:technical}
    Let $q_1, q_2 \in \Flat$ and suppose $\cyl(q_2) \setminus \cyl(q_1) \neq \emptyset$. For any $\alpha \in \cyl(q_2) \setminus \cyl(q_1)$, there exists another simple closed curve $\beta$ such that:
    \begin{itemize}
        \item $\beta$ intersects $\alpha$ transversely;
        \item $\beta \in \cyl(q_2)$;
        \item the Dehn twisted curve $T_\alpha\beta$ satisfies 
        \[ \ell_{q_1}(T_\alpha\beta) = \ell_{q_1}(\beta) + i(\alpha,\beta)\ell_{q_1}(\alpha). \]
    \end{itemize}
\end{lemma}

Using this result, we proceed to the proof of the main result of this subsection.

\begin{proof}[Proof of Proposition \ref{p:technical}]
    Suppose for the sake of contradiction that $\cyl(q_2) \not\subset \cyl(q_1)$. Then there exists a simple closed curve $\alpha \in \cyl(q_2) \setminus \cyl(q_1)$. Lemma \ref{l:technical} provides a curve $\beta \in \cyl(q_2)$ such that 
    \[ \ell_{q_1}(T_\alpha\beta) = \ell_{q_1}(\beta) + i(\alpha,\beta)\ell_{q_1}(\alpha). \] 
    Because $\alpha$ is a cylinder curve in $q_2$, Proposition \ref{p:length_Dehn} yields the strict inequality: 
    \[ \ell_{q_2}(T_\alpha\beta) < \ell_{q_2}(\beta) + i(\alpha,\beta)\ell_{q_2}(\alpha). \]
    By Corollary \ref{c:cylinder_same_length}, since $\alpha, \beta \in \cyl(q_2)$, we have $\ell_{q_1}(\alpha) = \ell_{q_2}(\alpha)$ and $\ell_{q_1}(\beta) = \ell_{q_2}(\beta)$. Combining these relations, we conclude that 
    \[ \ell_{q_2}(T_\alpha\beta) < \ell_{q_1}(T_\alpha\beta). \] 
    However, $T_\alpha\beta$ is a simple closed curve, so this strict inequality contradicts the initial hypotheses of Theorem \ref{main theorem}.
\end{proof}

We now prepare for the proof of Lemma \ref{l:technical}. The argument relies on a construction from the proof of \cite[Proposition 18]{duchin2010length}, which we refer to as a \emph{surgery curve}. A surgery curve is a specific piecewise geodesic representative of a Dehn twist in a flat metric, constructed as follows: first, take the geodesic representatives of $\alpha$ and $\beta$ with respect to the flat metric $q$. Then, cut along $\beta$ and insert a copy of $\alpha$ at each intersection. This surgery is performed via a ``left twist,'' meaning that when $\beta$ meets $\alpha$ at a transverse intersection, the curve turns left to merge onto $\alpha$, travels along the entirety of $\alpha$, and finally turns right to leave $\alpha$ and resume its path along $\beta$. If the transverse intersection consists of a concatenation of saddle connections (as described in Proposition \ref{p.transverse_intersection}), we make the additional convention that the curve makes turns at the midpoint of this segment. Note that the concept of turning left or right depends solely on the orientation of the surface, not on the chosen orientations of the curves $\alpha$ and $\beta$. Examples of surgery curves are illustrated in Figure \ref{fig:surgery_curve}.

\begin{figure}[htbp]
    \centering
    \includegraphics[width=0.5\linewidth]{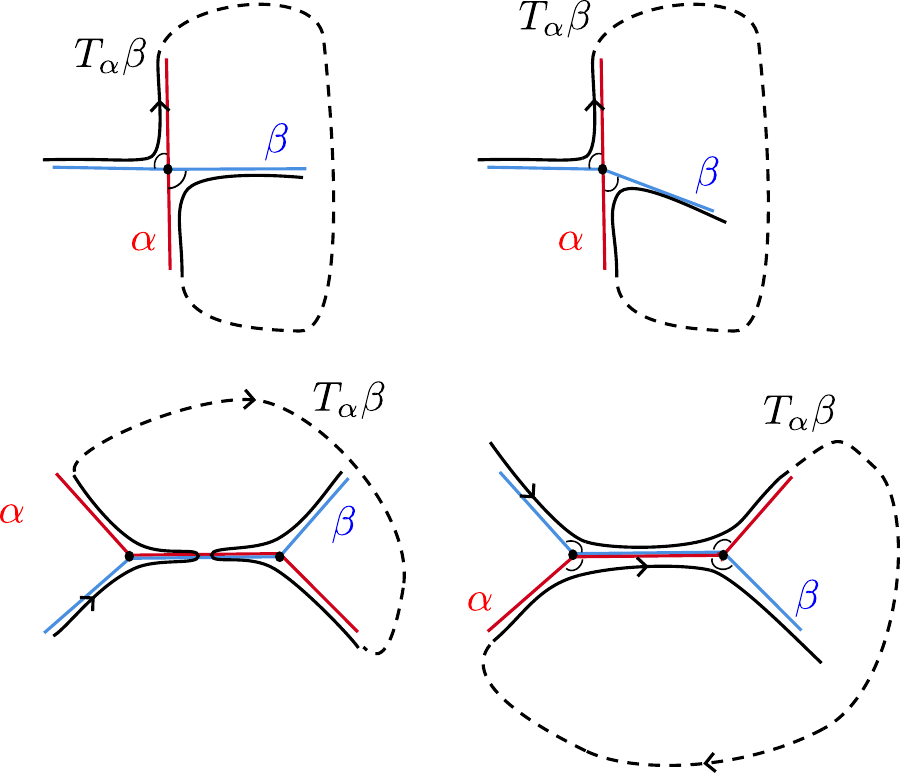}
    \caption{Upper left: $\alpha$ is a cylinder curve; the marked angles are strictly less than $\pi$, so the surgery curve is not a geodesic. \\
    Upper right: $\alpha$ is a singular closed geodesic and $\alpha\cap\beta$ is a cone point. The surgery curve is a geodesic if and only if the marked angles are at least $\pi$. \\
    Lower left: $\alpha$ is a singular closed geodesic and $\alpha\cap\beta$ is a concatenation of saddle connections. The curve $\beta$ turns right to enter the common saddle connections, meaning the left twist creates an overlap (backtracking). The surgery curve is not a geodesic. \\
    Lower right: $\alpha$ is a singular closed geodesic and $\alpha\cap\beta$ is a concatenation of saddle connections. The curve $\beta$ turns left to enter the common saddle connections, yielding a surgery curve that is a geodesic.}
    \label{fig:surgery_curve}
\end{figure}

Surgery curves offer two primary advantages. First, the length of this specific representative (which may not be the minimal length in its free homotopy class) is explicit: the length of the surgery curve representing $T_\alpha\beta$ is exactly $\ell_q(\beta) + i(\alpha,\beta)\ell_q(\alpha)$. This immediately provides an upper bound for the length of the true geodesic representative, yielding the inequality:
\[ \ell_q(T_\alpha\beta) \le \ell_q(\beta) + i(\alpha,\beta)\ell_q(\alpha). \]

Second, it is straightforward to determine whether a given surgery curve is a geodesic. Because the surgery curve is constructed as a piecewise geodesic, whether it is a true geodesic depends entirely on the angles formed at the surgery junctions. Recall from Proposition \ref{p.transverse_intersection} that if two curves $\alpha$ and $\beta$ intersect transversely, each connected component of their intersection is either an isolated point or a concatenation of saddle connections. If $\alpha$ is a cylinder curve, only the isolated point case occurs; consequently, the surgery curve always forms angles strictly less than $\pi$ at the intersection points, meaning it is never a geodesic. However, when $\alpha$ is a singular closed geodesic, the angle condition can sometimes be satisfied, allowing the surgery curve to be a true geodesic (see Figure \ref{fig:surgery_curve} for illustrations of these different cases).

The following result provides a method for constructing a curve $\beta$ such that the surgery curve of $T_\alpha\beta$ is a geodesic, provided $\alpha$ is a singular closed geodesic. The proof is essentially based on \cite[Lemma 20]{duchin2010length}, with modifications to suit our purposes. Note that this result does not require the hypotheses of Theorem \ref{main theorem}.

\begin{lemma}[{\cite[Lemma 20]{duchin2010length}}]\label{l:convergence_singular}
    Suppose $\alpha$ is a singular closed geodesic. Let $\{\beta_n\}_{n=1}^\infty$ be a sequence of simple closed geodesics intersecting $\alpha$ transversely. Choose lifts $\til\alpha$ and $\til\beta_n$ of $\alpha$ and $\beta_n$, respectively, such that each $\til\beta_n$ intersects $\til\alpha$ transversely. Orient these lifts so that each $\til\beta_n$ crosses $\til\alpha$ from left to right, and denote their respective endpoints by $\alpha^\pm$ and $\beta_n^\pm$. If $\beta_n^\pm \to \alpha^\pm$ as $n \to \infty$, then for all sufficiently large $n$, the surgery curve representing $T_\alpha\beta_n$ is a geodesic.
\end{lemma}

The choice of curve orientations here is purely for the convenienece of statement.

\begin{proof}
    As established in our previous discussion, given two transversely intersecting simple closed curves $\alpha$ and $\beta$, the surgery curve for $T_\alpha\beta$ is a geodesic if and only if each connected component of $\alpha \cap \beta$ falls into one of the two right-hand cases illustrated in Figure \ref{fig:surgery_curve}. Namely, the component must either be a single cone point where the two marked angles are at least $\pi$, or it must be a concatenation of saddle connections where $\beta$ turns left to merge onto the shared segment.

    The intersection of $\til\alpha$ and $\til\beta_n$ represents just one connected component of $\alpha \cap \beta_n$. We will first focus on this specific intersection and then use it to constrain the behavior of all other connected components.

    Because $\alpha$ is a singular geodesic, there exist two (not necessarily distinct) cone points $p_l, p_r$ on $\alpha$ such that the left angle formed by adjacent saddle connections at $p_l$, and the right angle formed at $p_r$, are strictly greater than $\pi$. Each of these two cone points has countably many lifts on $\til\alpha$. Choose lifts $\til p_l$ and $\til p_r$ such that $\til p_l$ is closer to $\alpha^-$ than $\til p_r$. Let $\til\gamma_l$ be a half-infinite geodesic ray based at $\til p_l$ such that the left angle between $\til\gamma_l$ and the geodesic ray $\til p_l\alpha^+$ is exactly $\pi$. Let $\gamma_l^-$ denote the ideal endpoint of $\til\gamma_l$, and let $A^-$ be the subarc of the circle at infinity bounded by $\gamma_l^-$ and $\alpha^-$. Similarly, construct the half-infinite geodesic ray $\til\gamma_r$ based at $\til p_r$ with ideal endpoint $\gamma_r^+$, defining the corresponding subarc $A^+$. This setup is illustrated in Figure \ref{fig:endpoints_convergence}.

    \begin{figure}[htbp]
        \centering
        \includegraphics[scale=0.65]{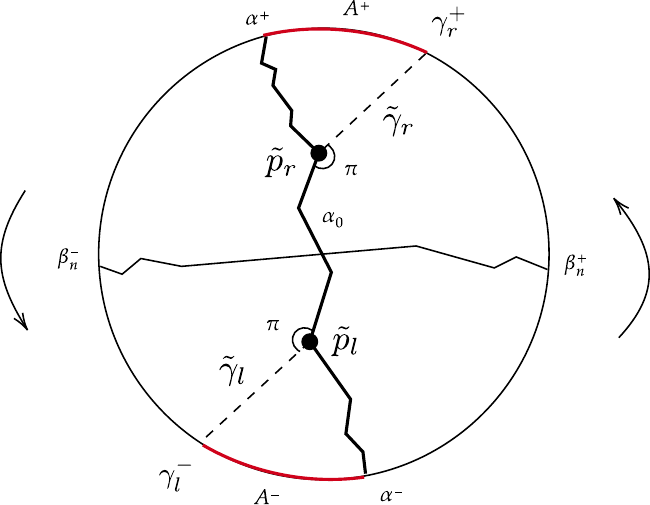}
        \caption{The endpoints of $\til\beta_n$ eventually fall into the arcs $A^\pm$.}
        \label{fig:endpoints_convergence}
    \end{figure}

    By construction, the concatenation of the three geodesic segments $\gamma_l^-\til p_l$, $\til p_l\til p_r$, and $\til p_r\alpha^+$ is itself a geodesic. The same is true for the concatenation of $\alpha^-\til p_l$, $\til p_l\til p_r$, and $\til p_r\gamma_r^+$. Because we assumed $\beta_n^\pm \to \alpha^\pm$, the endpoints $\beta_n^\pm$ will eventually fall strictly within $A^\pm$ for all sufficiently large $n$. Any geodesic with endpoints in $A^-$ and $A^+$ must contain the segment $\til p_l\til p_r$, because it cannot transversely intersect the bi-infinite geodesics with endpoints $\{\gamma_l^-, \alpha^+\}$ or $\{\alpha^-, \gamma_r^+\}$. Thus, for large $n$, $\til\beta_n$ intersects $\til\alpha$ along a sequence of saddle connections and turns left to enter them. This perfectly models the lower-right case in Figure \ref{fig:surgery_curve}.

    We now extend this argument to the other connected components of $\alpha \cap \beta_n$. Let $g$ be the deck transformation corresponding to $\alpha$; that is, $g$ fixes the endpoints $\alpha^\pm$ and acts on $\til\alpha$ by a translation of length $\ell_q(\alpha)$. Any other connected component of $\alpha \cap \beta_n$ can be represented by the intersection of $\til\alpha$ with some other lift of $\beta_n$, which we denote by $\til\beta_n'$. Because $\beta_n$ is simple, its lifts cannot intersect each other transversely. We may therefore choose $\til\beta_n'$ such that it lies entirely within the region bounded by $\til\beta_n$ and $g \cdot \til\beta_n$. When $n$ is sufficiently large, the endpoints of both $\til\beta_n$ and $g \cdot \til\beta_n$ lie inside $A^\pm$, forcing the endpoints of $\til\beta_n'$ to lie inside $A^\pm$ as well. Applying our previous logic, $\til\beta_n'$ must also turn left to enter its common saddle connections with $\til\alpha$. 
    
    Since the number of connected components of $\alpha \cap \beta_n$ is finite, for $n$ sufficiently large, all intersections are modeled by the lower-right case in Figure \ref{fig:surgery_curve}. Consequently, the surgery curve of such a $\beta_n$ is a geodesic.
\end{proof}

We are now ready to prove Lemma \ref{l:technical}, thereby completing the proof of Proposition \ref{p:technical}.

\begin{proof}[Proof of Lemma \ref{l:technical}]
    Lemma \ref{l:convergence_singular} provides a sufficient condition for a surgery curve to be a geodesic. Note that endpoint convergence is a topological property of the Gromov boundary, independent of the chosen locally $\mathrm{CAT}(0)$ metric. Thus, if a sequence of curves $(\beta_n)$ satisfies the hypotheses of Lemma \ref{l:convergence_singular} with respect to the metric $q_2$, evaluating this convergence with respect to $q_1$ preserves the limits of their endpoints.

    By assumption, the curve $\alpha$ is a cylinder curve in the metric $q_2$. Recall that a cylinder curve is uniquely determined by any unit vector tangent to it. Let $V_c$ be the set of unit vectors in the tangent bundle $T^1(\Sigma\setminus P)$ that generate cylinder curves. Any vector in $V_c$ whose basepoint lies within the maximal $q_2$-cylinder containing $\alpha$, and whose direction is distinct from that of $\alpha$, will generate a cylinder curve intersecting $\alpha$ transversely. By \cite{boshernitzan1998periodic}, the set $V_c$ is dense in $T^1(\Sigma\setminus P)$. Therefore, there exists a sequence of unit vectors $(v_n)_{n=1}^\infty \subset V_c$ such that:
    \begin{itemize}
        \item the basepoint of each $v_n$ lies in the maximal cylinder containing $\alpha$;
        \item the direction of each $v_n$ is distinct from that of the $q_2$-geodesic representative of $\alpha$;
        \item the sequence $(v_n)$ converges to a tangent vector of the $q_2$-geodesic representative of $\alpha$.
    \end{itemize}
    Let $\beta_n$ denote the cylinder curve generated by $v_n$. It remains to show that there exists a lift of each $\beta_n$ whose endpoints converge to $\alpha^\pm$.

    By passing to a subsequence if necessary, we may assume there exists a sequence of intersection points $p_n \in \alpha \cap \beta_n$ converging to a point $p \in \alpha$. Choose lifts $\til p_n$ in the universal cover $\til\Sigma$ such that they converge to a limit point $\til p$. This limit $\til p$ lies on a lift of $\alpha$, which we denote by $\til\alpha$. Let $\til\beta_n$ be the lift of $\beta_n$ intersecting $\til\alpha$ at $\til p_n$. We will demonstrate that $\beta_n^+ \to \alpha^+$, and the convergence of the backward endpoints $\beta_n^- \to \alpha^-$ follows in the same way. We abuse notation by letting $\til\alpha, \til\beta_n \colon [0,\infty) \to \til\Sigma$ denote the forward geodesic rays representing $\alpha^+$ and $\beta_n^+$, parameterized by arc length and originating at $\til p$ and $\til p_n$, respectively.
    
    A standard neighborhood basis for $\alpha^+$ in the boundary $\partial\til\Sigma$ is given by the sets:
    \[
    V(T,\epsilon) = \left\{ \gamma^+ \mid \text{geodesic } \gamma \colon [0,\infty) \to \til\Sigma,\ \gamma(0)=\til p, \ d(\til\alpha(t),\gamma(t)) \le \epsilon \text{ for all } t \in [0,T] \right\}.
    \]
    We aim to show that for any fixed $T$ and $\epsilon$, the endpoint $\beta_n^+$ lies in $V(T,\epsilon)$ for all sufficiently large $n$. 
    
    For a fixed $T$, because the directions of $(v_n)$ converge to that of $\alpha$, the initial segment $\til\beta_n|_{[0,T]}$ is contained in a flat Euclidean strip neighborhood of $\til\alpha$ for sufficiently large $n$. Therefore, we have the distance bound:
    \[
    d(\til\alpha(t),\til\beta_n(t)) \le t \sin\theta_n + t(1-\cos\theta_n) + d(\til p,\til p_n) \quad \text{for all } t \in [0,T],
    \]
    where $\theta_n \in [0,\pi/2]$ is the angle between the directions of $\alpha$ and $\beta_n$ (see Figure \ref{fig:asymptotic}).

    \begin{figure}[htbp]
        \centering
        \includegraphics[width=0.6\textwidth]{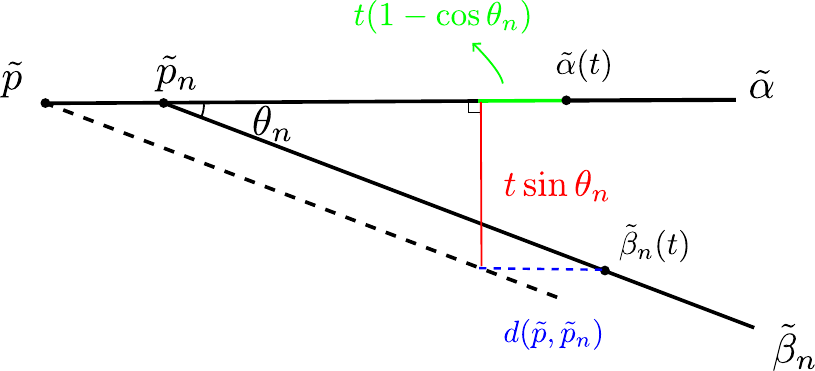}
        \caption{The segments $\til\alpha|_{[0,T]}$ and $\til\beta_n|_{[0,T]}$ lie within the same Euclidean chart.}
        \label{fig:asymptotic}
    \end{figure}

    Let $\til\beta_n' \colon [0,\infty) \to \til\Sigma$ be the geodesic ray pointing to $\beta_n^+$ with initial point $\til\beta_n'(0) = \til p$. Since $\til\beta_n$ and $\til\beta_n'$ are $\mathrm{CAT}(0)$ rays asymptotic to the same endpoint $\beta_n^+$, the distance function $t \mapsto d(\til\beta_n(t),\til\beta_n'(t))$ is convex and uniformly bounded. Thus, this function must be non-increasing, which in particular implies $d(\til\beta_n(t),\til\beta_n'(t)) \le d(\til\beta_n(0),\til\beta_n'(0)) = d(\til p,\til p_n)$. Combining these inequalities yields:
    \begin{align*}
        d(\til\alpha(t),\til\beta_n'(t)) &\le d(\til\alpha(t),\til\beta_n(t)) + d(\til\beta_n(t),\til\beta_n'(t)) \\
        &\le t \sin\theta_n + t(1-\cos\theta_n) + 2d(\til p,\til p_n) \quad \text{for all } t \in [0,T].
    \end{align*}
    For any fixed $\epsilon > 0$, since $\theta_n \to 0$ and $\til p_n \to \til p$, the right-hand side can be made strictly less than $\epsilon$ for all sufficiently large $n$. Therefore, $\beta_n^+ \in V(T,\epsilon)$, establishing the desired endpoint convergence. 
\end{proof}

\subsection{The same $\SL_2\R$-orbit}

In this subsection, we prove that under the hypotheses of Theorem \ref{main theorem}, the quadratic differential representatives of the two flat metrics $q_1$ and $q_2$ must lie in the same $\SL_2\R$-orbit. This reduces the problem to the special case previously analyzed in Section \ref{sec:ex}.

We begin by recalling the concept of \emph{uniquely ergodic foliations}.

\begin{defn}
    A singular foliation on a closed surface $\Sigma$ is \emph{uniquely ergodic} if it admits a unique transverse measure up to global scaling.
\end{defn}

We can equivalently characterize unique ergodicity in terms of intersection numbers. Recall that a measured foliation is called \emph{minimal} if its structural decomposition (see Theorem \ref{thm:foliation_decomposition}) consists of a single open domain in which every leaf is dense. Two measured foliations are \emph{topologically equivalent} if their underlying singular foliations, obtained by forgetting the transverse measures, are equivalent up to isotopy and Whitehead moves.

\begin{lemma}[{\cite[Theorem 1.12]{reesAlternativeApproachErgodic1981}}]\label{lem:minimal_foliation}
    Let $\mc F$ and $\mc G$ be measured foliations on a closed surface $\Sigma$, and suppose $\mc F$ is minimal. Then $i(\mc F,\mc G)=0$ if and only if $\mc F$ and $\mc G$ are topologically equivalent.
\end{lemma}

\begin{cor}\label{cor:unique_ergodic}
    Let $\mc F$ be a measured foliation on a closed surface $\Sigma$. The following conditions are equivalent:
    \begin{enumerate}
        \item $\mc F$ is uniquely ergodic.
        \item For any measured foliation $\mc G$, $i(\mc F,\mc G)=0$ if and only if $\mc F=c\mc G$ for some constant $c > 0$.
    \end{enumerate}
\end{cor}

\begin{proof}
    First, observe that every uniquely ergodic measured foliation is necessarily minimal; otherwise, one could construct distinct, non-proportional transverse measures by scaling different components by different factors, or by altering the measure on a cylinder domain. 
    
    $(1)\implies (2)$: If $i(\mc F,\mc G)=0$, then $\mc F$ and $\mc G$ are topologically equivalent by Lemma \ref{lem:minimal_foliation}. Because $\mc F$ is uniquely ergodic, any measured foliation topologically equivalent to it must share its transverse measure up to scaling, meaning $\mc F = c\mc G$ for some $c>0$.

    $(2)\implies (1)$: Suppose, for the sake of contradiction, that $\mc F$ is not uniquely ergodic. Then we can assign a distinct, non-proportional transverse measure to its underlying singular foliation, yielding a new measured foliation $\mc F'$. By construction, $\mc F$ and $\mc F'$ are topologically equivalent, so $i(\mc F, \mc F')=0$. However, $\mc F' \neq c\mc F$ for any constant $c>0$, which violates condition (2).
\end{proof}

Let $\UE(q)$ denote the set of directional measured foliations of $q$ that are uniquely ergodic, and let $\PUE(q)$ be its projectivization. These are subsets of $\mc{MF}(\Sigma)$ and $\mc{PMF}(\Sigma)$, respectively. Because almost every direction gives rise to a uniquely ergodic foliation \cite{kerckhoffErgodicityBilliardFlows1986}, both $\UE(q)$ and $\PUE(q)$ contain uncountably many elements. We are now in a position to state and prove the main result of this subsection.

\begin{prop}\label{p:same_orbit}
    Let $q_1$ and $q_2$ be two quadratic differentials on the same topological surface $\Sigma$. If $\cyl(q_2) \subset \cyl(q_1)$, then $q_2 = A q_1$ for some matrix $A \in \SL_2\R$.
\end{prop}

Duchin--Leininger--Rafi \cite[Lemma 22]{duchin2010length} demonstrated that two quadratic differentials sharing the exact same cylinder set must lie in the same $\SL_2\R$-orbit. A similar argument applies to our setting. We include the proof here for completeness.

\begin{proof}[Proof of Proposition \ref{p:same_orbit}]
    Let $\mc F_\theta \in \UE(q)$, and let $(\alpha_i)_{i=1}^\infty \subset \cyl(q)$ be a sequence of cylinder curves with respective directions $\theta_i$ converging to $\theta$. Because the space $\mc{PMF}(\Sigma)$ is compact, $(\alpha_i)_{i=1}^\infty$ has a subsequence converging projectively to some measured foliation $\mu$; that is, there exists a sequence of scalars $t_i \in \R_+$ such that $(t_i\alpha_i)_{i=1}^\infty$ converges to $\mu$. Note that $\alpha_i$ is a sub-foliation of $\mc F_{\theta_i}$, meaning $i(\alpha_i, \mc F_{\theta_i}) = 0$. By the continuity of the intersection form, we have
    \[ i(\mu, \mc F_\theta) = \lim_{i\to\infty} i(t_i\alpha_i, \mc F_{\theta_i}) = 0. \]
    By Corollary \ref{cor:unique_ergodic}, because $\mc F_\theta$ is uniquely ergodic, $\mu$ and $\mc F_\theta$ must coincide as elements in $\PUE(q)$. Furthermore, the directions of cylinder curves are dense in $\mathbb{S}^1$ \cite[Theorem 2]{masurClosedTrajectoriesQuadratic1986}. Thus, every measured foliation in $\PUE(q)$ can be projectively approximated by simple closed cylinder curves in $\cyl(q)$.

    Now, let $\mc F \in \PUE(q_2)$. By the density established above, there exists a sequence of simple closed curves $(\alpha_i)_{i=1}^\infty \subset \cyl(q_2)$ converging projectively to $\mc F$. Because $\cyl(q_2) \subset \cyl(q_1)$, the $q_1$-geodesic representatives of $\alpha_i$ are also cylinder curves. Let $\theta$ be the limit of a convergent subsequence of the $q_1$-directions of $\alpha_i$. Applying the exact same argument as above to the metric $q_1$, we obtain $i(\mc F, \mc F^{q_1}_\theta) = 0$. Since $\mc F$ is uniquely ergodic, $\mc F^{q_1}_\theta$ must agree with $\mc F$. Consequently, we have the inclusion $\PUE(q_2) \subset \PUE(q_1)$.

    As noted previously, almost every direction yields a uniquely ergodic foliation, so $\PUE(q_2)$ contains uncountably many elements. Choose two distinct foliations $\mu, \nu \in \PUE(q_1) \cap \PUE(q_2) = \PUE(q_2)$. There exist matrices $M, M' \in \SL_2\R$ such that $\mu$ and $\nu$ are the horizontal and vertical foliations of $M q_1$ and $M' q_2$, respectively. Because $M q_1$ and $M' q_2$ share the same horizontal and vertical projective foliations $\mu$ and $\nu$, they determine the same Teichmüller geodesic \cite{gardiner1991extremal}. Thus, there exists a diagonal matrix $\Lambda_t = \mathrm{diag}(e^t, e^{-t})$ such that $M q_1 = \Lambda_t M' q_2$. Rearranging this equation yields
    \[ q_1 = M^{-1}\Lambda_t M' q_2 \]
    where $M^{-1}\Lambda_t M' \in \SL_2\R$.
\end{proof}

Summarizing all the results above, we can now complete the proof of Theorem \ref{main theorem}.

\begin{proof}[Proof of Theorem \ref{main theorem}]
    Let $q_1$ and $q_2$ denote the quadratic differentials representing the two flat metrics. By Proposition \ref{p:technical} and Proposition \ref{p:same_orbit}, they lie in the same $\SL_2\R$-orbit, meaning there exists a matrix $A \in \SL_2\R$ such that $q_2 = A q_1$. By Proposition \ref{p:baby_case}, if the largest singular value of $A$ is strictly greater than 1, then $K(q_2, q_1) > 0$. This implies the existence of a simple closed curve $\gamma$ strictly satisfying $\ell_{q_1}(\gamma) > \ell_{q_2}(\gamma)$, which contradicts our initial length inequality hypothesis. Thus, all singular values of $A$ must equal 1. It follows that $A \in \SO_2\R$, meaning $q_1$ and $q_2$ define the exact same flat metric.
\end{proof}

\section{Topology of the distance}\label{topology}

By Corollary \ref{cor:nondegenerate}, the function $K$ is a well-defined asymmetric metric on $\Flat$. In this section, we study the topological properties of this metric. Since $K$ is asymmetric, it induces two topologies on $\Flat$, called the \emph{left topology} and the \emph{right topology}. The basis of them are ``left balls" and ``right balls", respectively, which are defined for a basepoint $q\in\Flat$ and radius $R>0$ as follows:
\begin{align*}
    B_q^l(R)=\{t\in\Flat \mid K(t,q)\le R\}, \\
    B_q^r(R)=\{t\in\Flat\mid K(q,t)\le R\}.
\end{align*}
Furthermore, $\Flat$ inherits the topology of geodesic currents, which is characterized by Theorem \ref{thm:topology_current}. Our main goal in this section is to compare these three topologies.

\begin{prop}\label{continuous distance}
    $K:\Flat\times \Flat\to \R_{\ge 0}$ is continuous, where $\Flat$ is equipped with the topology of space of geodesic currents.
\end{prop}

We need the following result for the proof.

\begin{prop}[{\cite[Proposition 4]{bonahon1988geometry}}]
    Let $h$ be an auxiliary flat metric on $S$, and let $L_h$ be its Liouville current. Then the set 
    \[ \Curr_0(S) := \{\alpha \in \Curr(S) \mid i(\alpha,L_h)=1\} \]
    is a compact subset of $\Curr(S)$.
\end{prop}

\begin{cor}
    Viewing $\mc{MF}(S)$ as a subset of $\Curr(S)$, the set 
    \[ \mc{MF}_0(S) := \{\mc F \in \mc{MF}(S) \mid i(\mc F,L_h)=1\} \]
    is a compact subset of $\mc{MF}(S)$.
\end{cor}

$\mc{MF}_0(S)$ gives rise to a section of the projection $\mc{MF}(S)\to P\mc{MF}(S)$ by sending each projective measured foliation to the representative whose intersection with $L_h$ is $1$.

\begin{proof}[Proof of Proposition \ref{continuous distance}]
    Observe that 
    \[ K(q_1,q_2) = \log\sup_{\gamma\in\mc S}\frac{\ell_{q_2}(\gamma)}{\ell_{q_1}(\gamma)} = \log\sup_{\mc{F} \in \mc{MF}(S)}\frac{\ell_{q_2}(\mc F)}{\ell_{q_1}(\mc F)} = \log\sup_{\mc{F} \in \mc{MF}_0(S)}\frac{\ell_{q_2}(\mc F)}{\ell_{q_1}(\mc F)}. \]
    Here, the second equality follows from the density of weighted simple closed curves in $\mc{MF}(S)$, and the final equality holds because the length ratio is invariant under scaling.

    Let $(q_n,t_n)$ be a sequence in $\Flat\times\Flat$ converging to $(q,t)$. For each $\mc F\in\mc{MF}_0(S)$, define the functions
    \[ r_n(\mc F)=\frac{\ell_{t_n}(\mc F)}{i(L_{q_n},\mc F)} \quad \text{and} \quad r(\mc F)=\frac{\ell_{t}(\mc F)}{\ell_{q}(\mc F)}. \]
    The convergence $(q_n,t_n)\to(q,t)$ implies that $r_n$ converges to $r$ pointwise. Because the functions $r_n$ and $r$ are continuous on the compact set $\mc{MF}_0(S)$, this convergence is in fact uniform. Consequently, $K(q_n,t_n)=\log\sup_{\mc F} r_n(\mc F)$ converges to $K(q,t)$.
\end{proof}

\begin{cor}\label{c.open_balls}
    The left and right balls $B^l_q(R), B^r_q(R)$ are closed subsets of $\Flat$ with respect to the topology of geodesic currents.
\end{cor}

\begin{proof}
    By Proposition \ref{continuous distance}, $K(\cdot,q)$ and $K(q,\cdot)$ are continuous functions from $\Flat$ to $\R$, so the preimage of the closed set $[0,R]$ is closed.
\end{proof}

Our next step is to determine whether left and right balls are compact with respect to the topology of geodesic currents. Equivalently, for a fixed basepoint $q$ and a sequence of geodesic currents $(q_n)$ leaving every compact subset of $\Flat$, is it true that $K(q,q_n)\to\infty$ and $K(q_n,q)\to\infty$?

We recall the compactification of $\Flat$ detailed in \cite[\S 6]{duchin2010length}, which provides a concrete way to describe the divergence of flat metrics. The points in this compactification are called \emph{mixed structures}. A mixed structure on $\Sigma$ is a triple $(X,q,\lambda)$,
\begin{itemize}
    \item $X$ is an open, $\pi_1$-injective subsurface with negative Euler characteristic;
    \item $q$ is a unit-area flat structure on $X$ such that the boundary components of $X$ degenerate to cone points with angles that are integer multiples of $\pi$;
    \item $\lambda$ is a measured lamination whose support can be homotoped into $\Sigma\setminus X$.
\end{itemize} 
Unlike cone points in the interior of $X$, a marked point is allowed to have a cone angle of $\pi$.

Every mixed structure is associated with a geodesic current as follows. Recall that a measured foliation on a subsurface can be extended to a measured foliation on the entire surface. Thus, the directional foliation $\mc F_\theta$ of $q$ on the subsurface $X$ gives rise to an element in $\mc{MF}(\Sigma) \subset \Curr(\Sigma)$. The Liouville current $L_q$ of $q$ is defined exactly as in Theorem \ref{thm:liouville_current}, yielding a geodesic current on the total surface $\Sigma$. The geodesic current of a mixed structure $\eta=(X,q,\lambda)$ is then defined to be $L_\eta=L_q+\lambda$. Let $P\Curr(S)$ be the projective space of geodesic currents and $P\operatorname{Mix}(S)$ be the projective space of mixed structures on $S$.

\begin{thm}[{\cite[Theorem 5]{duchin2010length}}]\label{thm:degeneration}
    The closure of $\Flat$ in $P\Curr(S)$ is $P\operatorname{Mix}(S)$.
\end{thm}

A point in $P\operatorname{Mix}(S)\setminus \Flat$ corresponds to a mixed structure with $X$ being proper subsurface of $\Sigma$ or being empty. Let $q_n$ be a sequence in $\Flat$ converging to a mixed structure $\eta=(X,q,\lambda)$ in $P\Curr(S)$; that is, there exists a sequence $t_n\in \R_+$ such that $t_nL_{q_n}$ converges to $\eta$ as geodesic currents. Since $i(t_nL_{q_n},t_nL_{q_n})=(t_n^2\pi)/2\to i(L_\eta,L_\eta)$, the degeneration belongs to one of the following cases:
\begin{itemize}
    \item if $X=\emptyset$, in which case $\eta$ is a measured lamination, then $t_n\to 0$;
    \item if $X\ne\emptyset$, then we can assume $t_n=1$ for all $n$.
\end{itemize}

\begin{prop}\label{properness}
    Fix a basepoint $q_0$, then $K(\cdot,q_0):\Flat\to \R$ is a proper function with respect to the topology of geodesic currents, while $K(q_0,\cdot)$ is not proper.
\end{prop}

\begin{proof}
    Take a sequence $q_i\in\Flat$ such that there exists $t_n\in\R_+$ and $t_nL_{q_n}$ converges to a mixed structure $\eta=(X,q,\lambda)\in\mathrm{Mix}(S)\setminus \Flat$, i.e., $X$ is a proper subset of $S$. 
    
    Let's first deal with the left ball $B^l_q(R)$, or equivalently the function $K(\cdot,q)$. We need to show that for any fixed $q_0$, $K(q_n,q_0)\to\infty$. When $n$ is sufficiently large, no matter $X$ is empty or not, we may assume $t_n\le 1$. Thus, \[
    \sup_{\mc F\in\mc{MF}_0(S)}\frac{i(L_{q_0},\mc F)}{i(L_{q_n},\mc F)}\ge \sup_{\mc F\in\mc{MF}_0(S)}\frac{i(L_{q_0},\mc F)}{t_ni(L_{q_n},\mc F)}.
    \] 
    Because of the compactness of $\mc{MF}_0(S)$, the numerator of the right term is bounded away from $0$. As for the denominator, note that $i(t_nL_{q_n},\lambda)\to i(L_\eta,\lambda)=0$, where $\lambda$ is the laminar part of the mixed structure. Thus, the supremum of the ratio goes to infinity.

    As for $B^r_q(R)$, let's assume $q_n$ converge projectively to a mixed structure $\eta$ with $X\ne\emptyset$. As we mentioned above, we can assume $t_n=1$ for all $n$ and $L_{q_n}$ converges to $\eta$ as geodesic currents. Then for any fixed $q_0\in\Flat$, the ratio $\frac{i(\eta,\mc F)}{i(L_{q_0},\mc F)}$ is a well-defined continuous function over $\mc{MF}_0(S)$, so its supremum cannot be infinity. Note that \[
    \lim_{n\to\infty}\sup_{\mc{F}}\frac{i(L_{q_n},\mc F)}{i(L_{q_0},\mc F)}=\sup_{\mc{F}}\lim_{n\to\infty}\frac{i(L_{q_n},\mc F)}{i(L_{q_0},\mc F)}=\sup_{\mc F}\frac{i(\eta,\mc F)}{i(L_{q_0},\mc F)},
    \]which implies $\{K(q_0,q_n)\}$ is a bounded sequence and does not diverge.
\end{proof}

In the preceding proof concerning the right balls or the function $K(q_0,\cdot)$, we addressed one of the degeneration cases. For the remaining case, suppose we take a sequence $(q_n)$ converging projectively to a measured lamination $\lambda$, then the scaling factor $t_n\to 0$. For any curve $\gamma$ transverse to the measured lamination, the limit $\lim_{n\to\infty}t_n\ell_{q_n}(\gamma)=i(\lambda,\gamma)$ is strictly positive. Because $t_n\to 0$, it implies that $\ell_{q_n}(\gamma)\to\infty$. Meanwhile, since $\ell_{q_0}(\gamma)$ is a fixed positive constant, it follows that $K(q_0,q_n)\to\infty$. Roughly speaking, $\Flat$ is still of infinite diameter, but the space can be very ``thin" in certain directions.

The following result follows directly from Proposition \ref{continuous distance} and Proposition \ref{properness}.

\begin{cor}\label{cor:three topology}
    Let $q_n,q$ be points in $\Flat$. Then 
    \begin{itemize}
        \item $q_n\to q$ is equivalent to $K(q_n,q)\to 0$; that is, the topology of geodesic currents coincide with the left topology.
        \item $q_n\to q$ implies $K(q,q_n)\to 0$.
    \end{itemize} 
\end{cor}

It is interesting to ask whether $K(q,q_n)\to 0$ implies $q_n\to q$. When $K(q,q_n)\to 0$, then $q_n$ might stay in a compact neighborhood of $q$ or diverge to a mixed structure in the boundary. 
In the first case, by taking the limit of a convergent subsequence and the continuity of $K$, we can conclude that $q_n\to q$. But we don't know whether the second case can happen. 
We propose the following equivalent question.

\begin{ques}
    Let $q$ be a fixed point in $\Flat$. Is there a positive constant $C$ (which may depend on $q$) such that \[
    \inf K(q,\eta)>C\]
    where the infimum is taken over all mixed structures $\eta$ in the boundary of $\Flat$? 
\end{ques}

% \begin{cor}\label{cor:three topology}
%     The topology of $\Flat$ generated by $B^l_q(R)$ coincides with the topology of geodesic currents, while $B^r_q(R)$ generate a different topology.
% \end{cor}\jia{need to construct a mixed structure with zero distance.}

% \begin{proof}
%     The equivalence of the topology generated by $B^l_q(R)$ and the topology of geodesic currents is a direct consequence of the previous corollary.
% \end{proof}

The symmetrization of $K$ provides a symmetric metric, which is defined by \[d^{sym}(q_1,q_2)=K(q_1,q_2)+K(q_2,q_1).\]
Similarly, it induces a topology on $\Flat$ by balls of the form \[B^{sym}_q(R)=\{t\in\Flat\ |\ d^{sym}(t,q)=d^{sym}(q,t)\le R\}.\]

\begin{cor}
    The topology induced by $d^{sym}$ coincides with the topology of geodesic currents.
\end{cor}

\begin{proof}
    Observe that $B^{sym}_q(R) \subset B^l_q(R)$ for any $R > 0$, which implies that the topology induced by $d^{sym}$ is finer than the left topology. Conversely, by Corollary \ref{cor:three topology}, if $(q_n)$ is a sequence such that $K(q_n,q) \to 0$, then $q_n \to q$ in the topology of geodesic currents. This, in turn, yields $K(q,q_n) \to 0$, implying that $d^{sym}(q_n,q) \to 0$. Thus, the topology induced by $d^{sym}$ is also coarser than the topology of geodesic currents, completing the proof.
\end{proof}

\section{Flat metrics in the same conformal class}\label{ex extremal length}

In this section, we consider the case when two flat metrics $q_1$ and $q_2$ are in the same conformal class. We can construct infinitely many closed curves $\gamma$ such that$\ell_{q_2}(\gamma)/\ell_{q_1}(\gamma)>1$.

We start with some background on extremal length. For a more detailed introduction, we refer the reader to \cite{jenkins1957existence,strebel1966quadratische,kerckhoffAsymptoticGeometryTeichmuller1980}.

\begin{defn}
    Given a Riemann surface $X$ and a simple closed curve $\gamma\subset X$, the \emph{extremal length} of $\gamma$, which is denoted by $\EL_X(\gamma)$, is defined to be 
    \[
    EL_X(\gamma)=\sup_g \frac{\ell_g^2(\gamma)}{\Area(X,g)}
    \]
    where $g$ ranges over all conformal metrics with respect to the conformal structure of $X$, and $\ell_g(\gamma)$ is the infimum of lengths of curves homotopic to $\gamma$.
\end{defn}

The following result was proved by Jenkins \cite{jenkins1957existence} and Strebel \cite{strebel1966quadratische}. Here we refer to the statement given by Kerckhoff.

\begin{thm}[{\cite[Theorem 3.1]{kerckhoffAsymptoticGeometryTeichmuller1980}}]\label{uniqueness for scc}
    For any simple closed curve $\gamma\subset X$, there is a unique quadratic differential $q\in Q^1(X)$ whose horizontal foliation is a cylinder with $\gamma$ being the core curve. Moreover, the flat metric induce by $q$ is the unique metric up to scaling which realizes the supremum.
\end{thm}

The uniqueness of the extremal length metric implies the following corollary.

\begin{cor}
    Let $q$ be a quadratic differential whose horizontal foliation is a cylinder with core curve $\gamma$. Denote also by $q$ the flat metric induced by it. Then for any flat metric $q'$ in the same conformal class with the same area, we have \[
    \ell_{q'}(\gamma)\le \ell_q(\gamma)\]
    with equality if and only if $q=q'$ as points in $\Flat$.
\end{cor}

For the general case, we need to consider the extremal length of measured foliations. In the same paper, Kerckhoff showed extremal length can be defined for measured foliations.

\begin{prop}[{\cite[Proposition 3]{kerckhoffAsymptoticGeometryTeichmuller1980}}]\label{extremal MF}
     There is a unique continuous extension of the extremal length function to $\mc{MF}(S)$ satisfying $EL_X(r\cdot)=r^2EL_X(\cdot),r\in\R_+$.
\end{prop}

Kerchhoff does not give a formula for the extremal length involving the length. We provide a formula in terms of the intersection number of geodesic currents, which is a generalization of Theorem \ref{uniqueness for scc}.

\begin{prop}\label{p.extremal_formula}
    Given a conformal structure $X$ on $\Sigma$, it holds true that \[
    \EL_X([\mc F])=\sup_q \frac{\paren*{i(L_q,\mc F)}^2}{\Area(q)}\]
    where $q$ ranges over flat metrics in the same conformal class. Moreover, the metric realizing the extremal length is unique up to scaling, which is the flat cone metric coming from the quadratic differential whose horizontal foliation is in the class $[\mc F]$.
\end{prop}

In fact, the metric can range over a larger family of metrics in the same conformal class. We postpone the proof of Proposition \ref{p.extremal_formula} in Appendix \ref{uniqueness extremal}.

\begin{rmk}
    Hubbard and Masur \cite{hubbardQuadraticDifferentialsFoliations1979} proved that for any measured foliation $\mc F$ and a conformal structure $X$, there exists a unique quadratic differential $q$ whose horizontal foliation is in the class $[\mc F]$. Note that the uniqueness of the extremal length metric in Proposition \ref{p.extremal_formula} is not a direct consequence of their result. If we take a sequence of simple closed curves $\gamma_n$ converging to $\mc F$, there might be flat metrics $q_n$ which do not realize the extremal length with an error term $\epsilon_n$ going to zero. The limit of $q_n$ might not be $q$.
\end{rmk}

\begin{cor}\label{c:foliation_length_conformal}
    Let $q$ be a quadratic differential and let $\mc F$ be a measured foliation in direction $\theta$. Denote also by $q$ the flat metric induced by it. Then for any flat metric $q'$ in the same conformal class with the same area, we have \[
    i(L_{q'},\mc F)\le i(L_q,\mc F)\]
    with equality if and only if $q=q'$ as points in $\Flat$. In particular, when $q$ and $q'$ are different metrics, any $\gamma$ close to the foliation $\mc F$ in $\mc{PMF}(\Sigma)$ satisfies $\ell_{q'}(\gamma)<\ell_q(\gamma)$.
\end{cor}

\begin{proof}
    When $\mc F$ is the horizontal foliation of $q$, the result follows from Proposition \ref{p.extremal_formula}. For a general measured foliation $\mc F$ in direction $\theta$, we can consider the quadratic differential $e^{-i\theta}q$ whose horizontal foliation is $\mc F$ and the flat metric induced by it is the same as $q$. Then the result follows from the previous case.
\end{proof}

Here is another corollary which is of independent interest.

\begin{cor}\label{c:intersection_inequality}
    Given a Riemann surface $X$ and any two quadratic differentials $q,q'\in Q^1(X)$, then \[
    i(L_{q},L_{q})\ge i(L_{q},L_{q'})\]
    with equality if and only if $q=cq'$ for some $c\in\C$ of unit modulus.
\end{cor}

\begin{proof}
    Let $\mc F_\theta$ be the measured foliation of $q$ in direction $\theta$. By Corollary \ref{c:foliation_length_conformal}, we have 
    \begin{equation*}
        i(L_q,L_{q'})=\frac{1}{2}\int_0^\pi i(L_{q'},\mc F_\theta)\,d\theta\le \frac{1}{2}\int_0^\pi i(L_q,\mc F_\theta)\,d\theta=i(L_q,L_q).
    \end{equation*}
    The equality holds if and only if $i(L_{q'},\mc F_\theta)=i(L_q,\mc F_\theta)$ for almost every $\theta$. By Corollary \ref{c:foliation_length_conformal}, this is equivalent to $q$ and $q'$ give the same flat metric.
\end{proof}

Bonahon \cite[Theorem 19]{bonahon1988geometry} proved a similar result for hyperbolic metrics. Given two hyperbolic metrics $g,h$ on a closed surface, we have \[
    i(L_{g},L_{h})\ge i(L_{g},L_{g})\]
with equality if and only if $g=h$. Note that the inequality is reversed compared to the flat case. Moreover, when the complex structure is varied, the inequality in Corollary \ref{c:intersection_inequality} does not hold. For example, consider a quadratic differential $q$ and let $q_t=\Lambda_tq$ where $\Lambda_t=\diag(t,t\inv)$. Let $\mc F_\theta$ be the measured foliation of $q$ in direction $\theta$. Note that $\mc F_\theta$ remains to be the measured foliation of $q_t$, but in a different direction. Then we have $i(L_{q_t},\mc F_\theta)=\Len_{q_t}(\mc F_\theta)$. By the computation in Section \ref{sec:ex}, the length measure along leaves of $\mc F_\theta$ is scaled by \[
\sqrt{t^2\cos^2\theta+t^{-2}\sin^2\theta}.\]
Note that as a measured foliation, the transverse measure of $\mc F_\theta$ is always the same, so \[
\Len_{q_t}(\mc F_\theta)=\sqrt{t^2\cos^2\theta+t^{-2}\sin^2\theta}\,\Len_q(\mc F_\theta).\]
Let $\phi(t,\theta)=\sqrt{t^2\cos^2\theta+t^{-2}\sin^2\theta}$, then \[
i(L_{q}, L_{q_t})=i(L_q,L_q)\int_0^\pi\phi(t,\theta)\,d\theta.\]
Note that \[
\diffp[2]{\phi(t,\theta)}{t}=\frac{2\sin^2\theta(3t^4\cos^2\theta+\sin^2\theta)}{t^3(t^4\cos^2\theta+\sin^2\theta)^\frac{3}{2}}\ge 0\quad \text{for }t>0,\]
and equality holds if and only if $\theta\in\{0,\pi\}$. Moreover, we have the symmetry that \[
\int_0^\pi\phi(t,\theta)\,d\theta=\int_0^\pi\phi(t\inv,\theta)\,d\theta.\]
Thus, $i(L_q,L_{q_t})$ is a strictly convex function of $t$ and the minimum is attained at $t=1$. Therefore, for any $t\ne 1$, we have \[i(L_q,L_q)<i(L_q,L_{q_t}).\]
\appendix

\section{Extremal length of a measured foliation}\label{appendix}

\subsection{Length of a measured foliation}

The space of measured foliations up to isotopy and Whitehead equivalence, denoted by $\mc{MF}(S)$, is homeomorphic to the space of measured laminations $\mc{ML}(S)$. In this subsection, we will discuss different ways to define the length of a measured foliation $\mc F$ and the relationship between them. 

Let us first fix the notation. Let $\Sigma$ be a surface with finitely marked points. Let $g$ be a metric which, in the complement of the marked points, is nonpositively curved with cone points of cone angle greater than $2\pi$. The marked points correspond to cone points of cone angles at least $\pi$. Denote the set of all such metrics by $\NPC(\Sigma)$.

Let $\lambda$ be a measured lamination with compact support on the complement of marked points in $\Sigma$. We need to fix the setting to realize it as a geodesic lamination. When the ambient metric is strictly negatively curved, the geodesic representative of each leaf of $\lambda$ is unique. In our setting, curvature being nonpositive allows multiple geodesic representatives of a leaf. If a leaf has two geodesic representatives, i.e., they are asymptotic in both forward and backward directions, then the leaf must be a closed curve in a flat cylinder by Proposition \ref{p.CAT(0)}. 

Denote by $\lambda_0$ the union of closed leaves in flat cylinders and denote by $\lambda_1$ the complement. Then $\lambda_0$ and $\lambda_1$ are disjoint sub-laminations. Each leaf in $\lambda_1$ has a unique geodesic representative with respect to $g$. For leaves in $\lambda_0$, we can choose the geodesic representative to be the core curve of the corresponding flat cylinder. With this convention, we have a well-defined notion of \emph{measured geodesic lamination} for metric $g$.

\begin{defn}
    Let $(\Sigma, g)$ be a surface where $g$ is nonpositively curved with finitely many cone points of cone angle greater than $2\pi$ and finitely many marked points, and let $(\lambda, \mu)$ be a measured lamination with compact support. We realize it as a geodesic lamination as above and consider the product measure $d\mu\times dl$, where $d\mu$ is the transverse measure of the lamination $\lambda$ and $dl$ is the length measure along leaves of the geodesic lamination $\lambda$. The measure is well-defined on the complement of the cone points and marked points. We define the length of $\lambda$ with respect to $g$ to be the integral of the measure $d\mu\times dl$ over the surface, denoted by $\ell_g(\lambda)$.
\end{defn}

Such a metric $g$ admits a Liouville current $L_g$, see \cite[\S 4]{constantineMarkedLengthSpectrum2018}. We briefly recall the construction of $L_g$. Consider the space of all bi-infinite geodesics in the universal cover $\til \Sigma$ which hit no cone points. It can be parametrized by the following charts: fix a regular geodesic arc $\alpha:[a,b]\to\til\Sigma$ which hits no cone points; let $D_\alpha\subset[a,b]\times (0,\pi)$ be the set of pairs $(t,\theta)$ such that there exists a regular geodesic $\gamma$ intersecting $\alpha$ transversely at $\alpha(t)$ with angle $\theta$. It is a subset of full Lebesgue measure, because the set of cone points is countable. We can define a measure $\widehat{L_g}$ by \[\frac{1}{2}\sin\theta\, d\theta\, dt\]
and extend by zero to the complement of $D_\alpha$. It is invariant under transition maps between charts, so it is globally well-defined. The Liouville current $L_g$ is the push-forward of $\widehat{L_g}$ to $G(\til\Sigma)$. The support of $\widehat{L_g}$ is the set of limit of regular geodesics, and since there are only countably many cone points, the set of regular geodesics is a full measure subset. 

\begin{prop}\label{p.length_agree_intersection}
    For any measured lamination $\lambda$ with compact support, we have $\ell_g(\lambda)=i(L_g,\lambda)$.
\end{prop}

\begin{proof}
    Notice that both $\ell_g(\cdot)$ and $i(L_g,\cdot)$ are additive and $\ell_g(\lambda_0)=i(L_g,\lambda_0)$, where $\lambda_0$ is a union of closed leaves. So we only need to prove the equality for $\lambda_1$, whose leaves are uniquely determined by endpoints.

    Recall that $\mc P \subset G(\til \Sigma) \times G(\til \Sigma)$ is the subspace consisting of linked pairs of endpoints. By definition, $i(L_g,\lambda_1)=\int_{\mc P/\pi_1(\Sigma)} (\lambda_1\times L_g)$. Here we make some further simplifications. First, we can restrict the integral to the subset $\mc P_0\subset \mc P$ consisting of points in the full measure subset, which are pairs of endpoints corresponding to leaves of $\lambda_1$ and regular geodesics. Second, by the construction of $\lambda_1$ and $\mc P_0$, it holds true that $\mc P_0$ is homeomorphic to \[
    \mc G_0=\set{(x,y)}{x\text{ is a regular geodesic, }y \text{ is a leaf of }\til{\lambda_1}, \ x\pitchfork y},\]
    where $\til{\lambda_1}$ is the lift of $\lambda_1$ to the universal cover, and $x\pitchfork y$ means that two geodesics intersect transversely in the normal sense, as $x$ is a regular geodesic.

    Let $U$ be an open foliation chart of $\til{\lambda_1}$ with no cone points in the interior. Define $\mc G_U\subset \mc G_0$ to be the set of pairs $(x,y)$ such that $x\cap y\in U$. To compute the total mass of $\mc G_U$, we can apply the Fubini theorem: fix a leaf $y$, then the $L_g$-mass of the set of regular geodesics intersecting $y\cap U$ is exactly the length of $y\cap U$; then integrate over all leaves of $\til{\lambda_1}$ intersecting $U$, we get the total mass of $\mc G_U$ is exactly the integral of the measure $d\mu\times dl$ over $U$. 

    There exists finitely many foliation charts $\{U_i\}$ which are open subsets of $\Sigma$ such that
    \begin{itemize}
        \item $\lambda_1\subset \bigcup_i \bar{U_i}$;
        \item $U_i\cap U_j=\emptyset$;
        \item $\bar{U_i}\cap\bar{U_j}$ is a segment which lies in the complement of $\lambda_1$ or intersects $\lambda_1$ transversely;
        \item there is no cone point in the interior of $U_i$.
    \end{itemize}
    The construction is motivated by \cite{kahnConformalSurfaceEmbeddings2022}. Since we deal with metrics with cone points, and foliation charts are used multiple times later, we include the details here. See Figure \ref{fig:tt_nhd} for an illustration.
    
    First, we can take a finite set of disjoint intervals $I_j\subset\Sigma$ such that they contain no cone points and every leaf of $\lambda_1$ intersects some $I_j$ transversely. It can be done because $\lambda_1$ is decomposed into finitely many minimal components, and it suffices to take one transverse interval for each minimal component. 
    
    Note that when there exist cone points, different leaves may meet at a cone point and travel along the same segment. To take this case into account, we put an additional transverse interval at each cone point.
    
    $\cup I_j$ divide leaves of $\lambda_1$ into groups of segments, such that segments in the same group are homotopic to each other relative to $I_j$. For each group, we can take a foliation chart $U_i$ homeomorphic to a rectangle such that one pair of opposite sides is contained in $\cup I_j$ and the other pair of opposite sides is contained in the complement of $\lambda_1$. By construction, the foliation charts satisfy the above properties.

    \begin{figure}[htbp]
        \centering
        \includegraphics[width=0.6\textwidth]{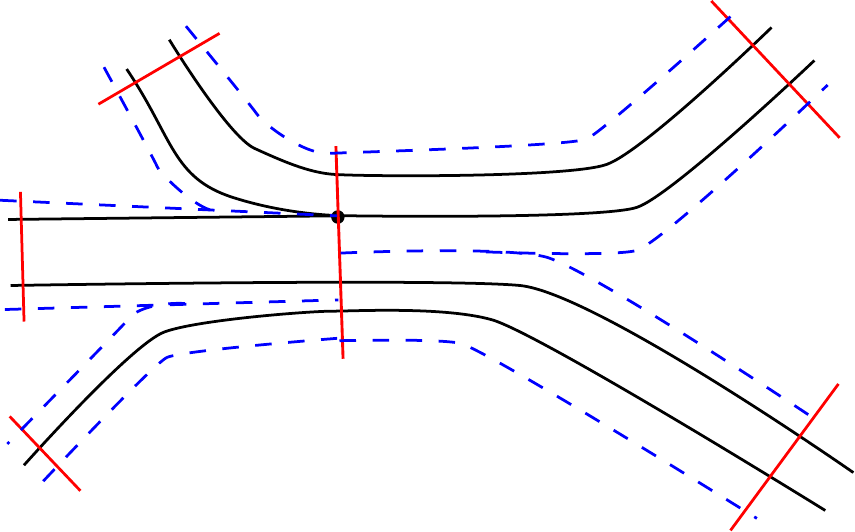}
        \caption{An example of foliation charts. $\{I_j\}$ are red segments transverse to the lamination, and blue dashed lines are in the complement of the lamination. Each foliation chart is enclosed by two red segments and two dashed blue lines.}
        \label{fig:tt_nhd}
    \end{figure}

    In fact, each $I_j$ and the minimal foliation component gives rise to a linear involution which is a generalization of interval exchange transformations, and every foliation chart corresponds to a subinterval of the linear involuation, see \cite[\S 2.1]{boissyDynamicsGeometryRauzy2009}. In particular, there are finitely many such foliation charts. Because of the finiteness, we may adjust the width of foliation charts so that for each $I_j$, charts on two sides have the same total width and match each other. Then the union of closure of foliation charts is a neighborhood of the lamination, denoted by $N(\lambda)$.

    Note that for a foliation chart $U_i$, the set $\mc G_{\partial U_i}$, which consists of pairs $(x,y)$ such that $x\cap y\in \partial U_i$, has measure zero, because a boundary segment either has zero transverse measure or zero length measure along leaves. Then the total mass of $\mc G_0$ is the sum of the total mass of $\mc G_{U_i}$, which is exactly the integral of the measure $d\mu\times dl$ over $\Sigma$. Thus, we have $i(L_g,\lambda_1)=\ell_g(\lambda_1)$.
\end{proof}

\begin{defn}
    Let $[\mc F]$ be an equivalent class of measured foliations and let $\mc F$ be a representative. 
    \begin{enumerate}
        \item Define the length of $\mc F$ with respect to $g$, denoted by $\Len_g(\mc F)$, to be the integral of the measure $d\mu\times dl$ over the surface, where $dl$ is the length measure along leaves of $\mc F$. 
        \item Define the length of $[\mc F]$ with respect to $g$, denoted by $\Len_g([\mc F])$, to be the infimum of $\Len_g(\mc F)$ over all representatives $\mc F\in [\mc F]$.
    \end{enumerate}
\end{defn}

Note that $\Len_g(\mc F)$ could be infinite. But we can always find a representative $\mc F$ such that $\Len_g(\mc F)$ is finite, for example, by taking a representative whose leaves are all piecewise rectifiable and locally of finite length. Therefore, $\Len_g([\mc F])$ is always finite.

When $\mc F$ is equivalent to a weighted simple closed curve, $\Len_g(\mc F)$ agrees with the definition of curve length. Motivated by the observation that the infimum of closed curve lengths is realized by a geodesic representative, we prove the following equality between different notions of length.

\begin{prop}\label{p.length_agree}
    For any measured lamination $\lambda$, let $\mc F_\lambda$ be its corresponding equivalent class of measured foliations. Then $\Len_g([\mc F_\lambda])=\ell_g(\lambda)$.
\end{prop}

The main idea is to prove two inequalities between two quantities. To prove the direction $\Len_g([\mc F_\lambda]) \ge \ell_g(\lambda)$, we first focus on the baby case where the measured foliation is equivalent to a weighted multicurve.

\begin{lemma}
    If there exists a representative $\mc F\in [\mc F_\lambda]$ whose leaves are all closed, then $\Len_g(\mc F)\ge\ell_g(\lambda)$.
\end{lemma}

\begin{proof}
    In this case, $\mc F$ can be decomposed into a union of cylinders. The leaves of $\lambda$ are the geodesic representatives of the core curves of the cylinders. Thus, the length of each leaf of $\lambda$ is less than or equal to the length of the corresponding leaf of $\mc F$. Therefore, we have $\Len_g(\mc F)\ge\ell_g(\lambda)$.
\end{proof}

\begin{lemma}\label{l.closed_foliation_inequality}
    For any measured lamination $\lambda$ and any representative $\mc F\in [\mc F_\lambda]$, we have $\Len_g(\mc F)\ge\ell_g(\lambda)$. In particular, $\Len_g([\mc F_\lambda]) \ge \ell_g(\lambda)$.
\end{lemma}

\begin{proof}
    We will construct a sequence of measured foliations $\mc F_n$ such that 
    \begin{itemize}
        \item $\mc F_n$ consists of closed leaves;
        \item $[\mc F_n]\to [\mc F_\lambda]$ in the topology of $\mc{MF}(\Sigma)$;
        \item $\Len_g(\mc F_n)\to\Len_g(\mc F)$.
    \end{itemize}
    Assume the existence of such a sequence. Let $\lambda_n$ be the measured lamination corresponding to $\mc F_n$, then $\lambda_n\to\lambda$ in the topology of $\mc{ML}(\Sigma)$. By the continuity of intersection number and Proposition \ref{p.length_agree_intersection}, we have $\lim_n\ell_g(\lambda_n)=\ell_g(\lambda)$. Then we can apply Lemma \ref{l.closed_foliation_inequality} and get \[
    \Len_g(\mc F)=\lim_{n\to\infty}\Len_g(\mc F_n)\ge \lim_{n\to\infty}\ell_g(\lambda_n)=\ell_g(\lambda).\]

    The main idea to construct the sequence $\mc F_n$ is to find a train track $\tau$ carrying $\mc F$ and adjust the weights of branches in $\tau$. Similar to construction in the proof of Proposition \ref{p.length_agree_intersection}, we can find transverse intervals $\{I_j\}$ and foliation charts $\{U_i\}$ such that the surface is covered by the union of their closures. For measured foliations, we don't need the additional intervals at cone points, as there are always leaves which are locally parallel to leaves through cone points. They give rise to a train track in the following way:
    \begin{itemize}
        \item each $I_j$ is collapsed to a point, which is a vertex of $\tau$;
        \item each foliation chart $U_i$ is collapsed to an edge, which is a branch of $\tau$;
        \item adjacent branches to the same $I_j$ are divided into two sets depending on which side of $I_j$ they are in, corresponding to incoming and outgoing directions;
        \item the width of $U_i$ is the weight of the branch.
    \end{itemize}

    \begin{figure}[htbp]
        \centering
        \includegraphics[width=0.4\textwidth]{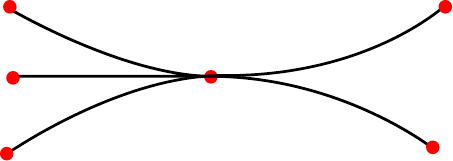}
        \caption{Train track of the foliation charts in Figure \ref{fig:tt_nhd}.}
    \end{figure}

    We can approximate the weight vector of the measured foliation $\mc F$ by a sequence of weight vectors with values in $\mb Q$, which gives rise to a sequence of equivalent classes of foliations $[\mc F_n]$. A representative of $[\mc F_n]$ can be constructed by using the same foliation charts, keeping the length of segments and adjusting the width. In other words, $\mc F_n$ has the same topological foliation charts as $\mc F$, but the transition functions between charts are changed according to the weights.

    By the choice of foliation charts, we have $\Len_g(\mc F_n)\to\Len_g(\mc F)$. On the other hand, by the choice of weight vector, each $\mc F_n$ consists of closed leaves, and $[\mc F_n]\to[\mc F]=[\mc F_\lambda]$.
\end{proof}

\begin{lemma}\label{l.foliation_approximation}
    For any measured lamination $\lambda$ and any $\epsilon>0$, there exists a representative $\mc F\in [\mc F_\lambda]$ such that $\Len_g(\mc F)\le \ell_g(\lambda)+\epsilon$. 
\end{lemma}

\begin{proof}
    Realize $\lambda$ as a measured geodesic lamination. Our goal is to smear out the complementary region of $\lambda$ without adding too much total length.

    We first do it in local foliation charts constructed in the proof of Proposition \ref{p.length_agree_intersection}. Let $f_i:U_\alpha\to R\cong I\times J$ be a foliation chart and $f_i(\lambda)\cap R$ is homeomorphic to $C\times J$, where $C\subset I$ is either a Cantor set or a discrete subset. The transverse measure $\mu$ of $\lambda$ is locally restricted to a measure on $I$ with support $C$. Then $f_i(\lambda)\cap R$ can be extended to a foliation $\mc F_i$ of the entire $R$ by disjoint geodesic segments. Let $\nu$ be a probability Lebesgue measure on $I$. Then $(\mc F_i, \mu+t\nu)$ is a genuine measured foliation in the local chart, for any $t\in\R_+$.

    Note that the length of leaves in $\mc F_i$ is bounded above by a constant determined by the foliation chart. Since there are only finitely many foliation charts, all these local foliations have leaves whose length is uniformly bounded above.

    Next we use train track to construct a foliation globally so that all local constructions match each other. Take the foliation charts of the measured geodesic lamination $\lambda$ and collapse them to get a train track as in the proof of Lemma \ref{l.closed_foliation_inequality}. The measured lamination gives rise to a weight vector $\vec{v}$ of the train track. We add a weight vector $t\vec{v}$ to it, which is realized by a Lebesgue measure in local charts as the construction above. After foliating every local charts, we get a measured foliation in $N(\lambda)$ which is the union of closure of foliation charts. Its leaves are piecewise geodesics.

    The final step is to smear out the complement of the union of foliation charts. Each connected component of the complement of $N(\lambda)$, denoted by $\Omega_i$, deformation retracts onto a spine, see \cite[Epilogue]{pennerCombinatoricsTrainTracks1992}. Choose a slightly fattened neighborhood $\Omega_i'$ of $\Omega_i$, then there exists a homeomorphism sending $N(\lambda)\setminus \partial N(\lambda)$ to the complement of the spines, which is identity on the complement of the union of $\Omega_i'$. Together with the spine, we get a measured foliation on the entire surface.

    \begin{figure}[htbp]
        \centering
        \includegraphics[width=0.6\textwidth]{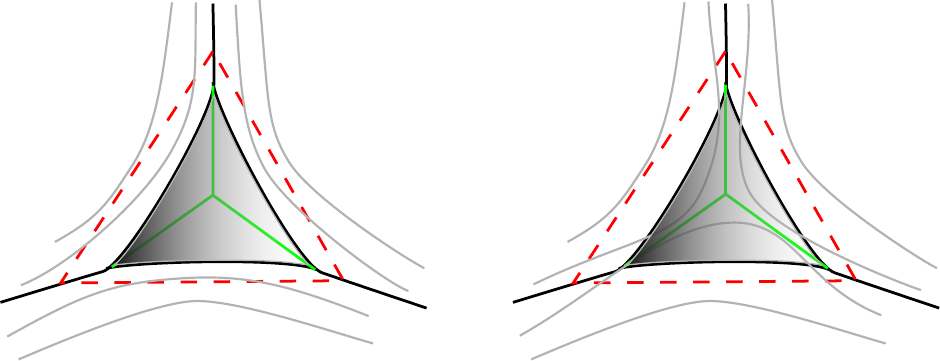}
        \caption{The shaded region $\Omega_i$ is a component of the complement of $N(\lambda)$, whose spine is colored in green. The fatten neighborhood $\Omega_i'$ is enclosed by red dashed lines. The foliation is extended to the entire surface by smearing out the complement.}
    \end{figure}

    It remains to analyze the length error between geodesic lamination and the foliation. By extending the lamination to a partial foliation in the neighborhood $N(\lambda)$, we add leaves with additional measure given by $t\vec{v}$. Recall that the length of leaves in each foliation chart is uniformly bounded above, so the length error tends to zero when $t$ goes to zero. In the step of smearing out the complement of $N(\lambda)$, the transverse measure is not changed, and leaves in $\Omega_i'\cap N(\lambda)$ become longer. Again the length difference of the deformation is bounded by constants determined by the complement component $\Omega_i$, and it is uniformly bounded above by the finiteness of the number of complement components. By shrinking the fattened region $\Omega_i'$, the error can be arbitrarily small. 
\end{proof}

\begin{proof}[Proof of Proposition \ref{p.length_agree}]
    It follows from Lemma \ref{l.closed_foliation_inequality} and Lemma \ref{l.foliation_approximation}.
\end{proof}

\subsection{Uniqueness of extremal metric}\label{uniqueness extremal}

Kerchhoff showed that there is a unique continuous extension of the extremal length from the set of simple closed curves to $\mc{MF}(\Sigma)$ \cite[Proposition 3]{kerckhoffAsymptoticGeometryTeichmuller1980}, which is denoted by $\EL_X([\mc F])$. Moreover, $\EL_X([\mc F])$ is equal to the area of the flat cone metric coming from the quadratic differential whose horizontal foliation is in the class $[\mc F]$.

In this subsection, we give another formula for the extremal length of measured foliations with the length function defined above, and prove the uniqueness of the metric realizing the extremal length.

\begin{prop}\label{p.extremal_formula_2}
    Given a conformal structure $X$ on $\Sigma$, it holds true that \[
    \EL_X([\mc F])=\sup_g \frac{\paren*{\Len_g\paren*{[\mc F]}}^2}{\Area(g)}\]
    where $g$ ranges over metrics in $\NPC(\Sigma)$ in the same conformal class of $X$. Moreover, the metric realizing the extremal length is unique up to scaling, which is the flat cone metric coming from the quadratic differential whose horizontal foliation is in the class $[\mc F]$.
\end{prop}

We need the following lemma.

\begin{lemma}\label{l.foliation_area}
    Let $q$ be a quadratic differential and $\mc F_q$ be its horizontal foliation. Then \[
    \Len_q(\mc F_q)=\Area(q).\]
\end{lemma}

\begin{proof}
    Choose coordinates $z$ such that locally $q=dz^2$. Then both the length measure along leaves and the transverse measure of $\mc F_q$ is $dz$. Thus, the length of $\mc F_q$ agrees with the area of the metric given by $q$.
\end{proof}

\begin{proof}[Proof of Proposition \ref{p.extremal_formula_2}]
    Let $[\mc F]$ be an equivalent class of measured foliations, and $q$ be the quadratic differential whose horizontal foliation $\mc F_q$ is in the class $\mc F$. Let $g$ be a metric in $\NPC(\Sigma)$ which is in the same conformal class of $q$, i.e., $g=\rho^2q$ for some $\rho$. Choose coordinates $z$ such that locally $q=dz^2$. Then the length measure along leaves of $\mc F_q$ under metric $g$ is $dl_g=\rho \,dz$, so we have 
    \begin{equation}\label{e.extremal}
        \begin{aligned}
            (\Len_g(\mc F_q))^2 &=\paren*{\int_\Sigma dl_g\,d\mu}^2=\paren*{\int_\Sigma\rho\,dz^2}^2\\
            &\le \paren*{\int_\Sigma\rho^2\,dz^2}\paren*{\int_\Sigma dz^2}\\
            &=\Area(g)\Area(q).
        \end{aligned}
    \end{equation}
    The inequality is given by Cauchy-Schwarz inequality. It follows immediately that \[
    \frac{\paren*{\Len_g\paren*{[\mc F]}}^2}{\Area(\Sigma, g)}\le\frac{\paren*{\Len_g\paren*{\mc F}}^2}{\Area(\Sigma, g)} \le \Area(q)=\EL_X([\mc F]).\]

    Moreover, by Lemma \ref{l.foliation_area}, we have \[
    \frac{\paren*{\Len_q\paren*{[\mc F]}}^2}{\Area(q)}=\Area(q),\]
    so the extremal length is realized by the metric $q$.

    When the equality holds in \ref{e.extremal}, then $\rho=c$ almost everywhere for some constant $c$. Thus, $g$ is a scaling of $q$.
\end{proof}

\bibliographystyle{alpha}
\bibliography{ref}

@article{duchin2010length,
  title={Length spectra and degeneration of flat metrics},
  author={Duchin, Moon and Leininger, Christopher J and Rafi, Kasra},
  journal={Inventiones mathematicae},
  volume={182},
  number={2},
  pages={231--277},
  year={2010},
  publisher={Springer}
}

@article{masurClosedTrajectoriesQuadratic1986,
  title = {Closed Trajectories for Quadratic Differentials with an Application to Billiards},
  author = {Masur, Howard},
  year = {1986},
  month = jun,
  journal = {Duke Mathematical Journal},
  volume = {53},
  number = {2},
  pages = {307--314},
  publisher = {Duke University Press},
  issn = {0012-7094, 1547-7398},
  doi = {10.1215/S0012-7094-86-05319-6},
  urldate = {2025-10-17}
}

@article{gardiner1991extremal,
  title={Extremal length geometry of Teichm{\"u}ller space},
  author={Gardiner, Frederick P and Masur, Howard},
  journal={Complex Variables and Elliptic Equations},
  volume={16},
  number={2-3},
  pages={209--237},
  year={1991},
  publisher={Taylor \& Francis}
}

@article{boshernitzan1998periodic,
  title={Periodic billiard orbits are dense in rational polygons},
  author={Boshernitzan, Michael and Galperin, Gregory and Kr{\"u}ger, Tyll and Troubetzkoy, S},
  journal={Transactions of the American Mathematical Society},
  volume={350},
  number={9},
  pages={3523--3535},
  year={1998}
}

@article{thurston1998minimal,
  title={Minimal stretch maps between hyperbolic surfaces},
  author={Thurston, William P},
  journal={arXiv preprint math/9801039},
  year={1998}
}

@article{bonahon1988geometry,
  title={The geometry of Teichm{\"u}ller space via geodesic currents},
  author={Bonahon, Francis},
  journal={Inventiones mathematicae},
  volume={92},
  number={1},
  pages={139--162},
  year={1988},
  publisher={Springer}
}

@article{hubbardQuadraticDifferentialsFoliations1979,
  title = {Quadratic Differentials and Foliations},
  author = {Hubbard, John and Masur, Howard},
  year = {1979},
  journal = {Acta Mathematica},
  volume = {142},
  number = {0},
  pages = {221--274},
  issn = {0001-5962},
  doi = {10.1007/BF02395062},
  urldate = {2025-09-15},
  langid = {english}
}

@article{kerckhoffAsymptoticGeometryTeichmuller1980,
  title = {The Asymptotic Geometry of Teichmuller Space},
  author = {Kerckhoff, Steven P.},
  year = {1980},
  journal = {Topology},
  volume = {19},
  number = {1},
  pages = {23--41},
  issn = {00409383},
  doi = {10.1016/0040-9383(80)90029-4},
  urldate = {2025-09-15},
  copyright = {https://www.elsevier.com/tdm/userlicense/1.0/},
  langid = {english}
}

@book{fathiThurstonsWorkSurfaces2012,
  title = {Thurston's Work on Surfaces},
  editor = {Fathi, Albert and Laudenbach, Fran{\c c}ois and Po{\'e}naru, Valentin},
  year = {2012},
  series = {Mathematical Notes},
  number = {48},
  publisher = {Princeton University Press},
  address = {Princeton},
  isbn = {978-0-691-14735-2 978-1-4008-3903-2},
  langid = {english}
}

@book{kapovich2001hyperbolic,
  title={Hyperbolic manifolds and discrete groups},
  author={Kapovich, Michael},
  volume={183},
  year={2001},
  publisher={Springer Science \& Business Media}
}

@incollection{strebel1984quadratic,
  title={Quadratic differentials},
  author={Strebel, Kurt},
  booktitle={Quadratic Differentials},
  pages={16--26},
  year={1984},
  publisher={Springer}
}

@incollection{masur2002rational,
  title={Rational billiards and flat structures},
  author={Masur, Howard and Tabachnikov, Serge},
  booktitle={Handbook of dynamical systems},
  volume={1},
  pages={1015--1089},
  year={2002},
  publisher={Elsevier}
}

@article{jenkins1957existence,
  title={On the existence of certain general extremal metrics},
  author={Jenkins, James A},
  journal={Annals of Mathematics},
  volume={66},
  number={3},
  pages={440--453},
  year={1957},
  publisher={JSTOR}
}

@incollection{strebel1966quadratische,
  title={{\"U}ber quadratische Differentiale mit geschlossenen Trajektorien und extremale quasikonforme Abbildungen},
  author={Strebel, Kurt},
  booktitle={Festband zum 70. Geburtstag von Rolf Nevanlinna: Vortr{\"a}ge, gehalten anl{\"a}$\beta$lich des Zweiten Rolf Nevanlinna-Kolloquiums in Z{\"u}rich vom 4.--6. November 1965},
  pages={105--127},
  year={1966},
  publisher={Springer}
}

@article{croke2004lengths,
  title = {Lengths and Volumes in {{Riemannian}} Manifolds},
  author = {Croke, Christopher B. and Dairbekov, Nurlan S.},
  year = {2004},
  month = oct,
  journal = {Duke Mathematical Journal},
  volume = {125},
  number = {1},
  issn = {0012-7094},
  doi = {10.1215/S0012-7094-04-12511-4},
  urldate = {2025-09-15},
  langid = {english}
}

@article{CONSTANTINE_SHRESTHA_WU_2025, 
  title={Sub-actions for geodesic flows on locally CAT($-1$) spaces}, 
  DOI={10.1017/etds.2025.10199}, 
  journal={Ergodic Theory and Dynamical Systems}, 
  author={CONSTANTINE, DAVID and SHRESTHA, ELVIN and WU, YANDI}, year={2025}, 
  pages={1–44}
}

@article{sauglam2024thurston,
  title={Thurston’s asymmetric metric on the space of singular flat metrics with a fixed quadrangulation},
  author={Sa{\u{g}}lam, {\.I}smail and Papadopoulos, Athanase},
  journal={L’Enseignement Math{\'e}matique},
  volume={70},
  number={1},
  pages={231--250},
  year={2024}
}

@article{lopes2005sub,
  title={Sub-actions for Anosov flows},
  author={Lopes, Artur O and Thieullen, Ph},
  journal={Ergodic Theory and Dynamical Systems},
  volume={25},
  number={2},
  pages={605--628},
  year={2005},
  publisher={Cambridge University Press}
}

@article{burgerCurrentsSystolesCompactifications2021,
  title = {Currents, Systoles, and Compactifications of Character Varieties},
  author = {Burger, M. and Iozzi, A. and Parreau, A. and Pozzetti, M. B.},
  year = 2021,
  month = dec,
  journal = {Proceedings of the London Mathematical Society},
  volume = {123},
  number = {6},
  pages = {565--596},
  issn = {0024-6115, 1460-244X},
  doi = {10.1112/plms.12419},
  urldate = {2025-09-15},
  langid = {english}
}

@article{shiLiouvilleCurrentHolomorphic2024,
  title = {The {{Liouville}} Current of Holomorphic Quadratic Differential Metrics},
  author = {Shi, Jiajun},
  year = 2024,
  month = jun,
  journal = {Geometriae Dedicata},
  volume = {218},
  number = {3},
  pages = {79},
  issn = {0046-5755, 1572-9168},
  doi = {10.1007/s10711-024-00928-w},
  urldate = {2025-09-15},
  langid = {english}
}

@book{bridsonMetricSpacesNonPositive1999,
  title = {Metric {{Spaces}} of {{Non-Positive Curvature}}},
  author = {Bridson, Martin R. and Haefliger, Andr{\'e}},
  editor = {Chern, S. S. and Eckmann, B. and De La Harpe, P. and Hironaka, H. and Hirzebruch, F. and Hitchin, N. and H{\"o}rmander, L. and Knus, M.-A. and Kupiainen, A. and Lannes, J. and Lebeau, G. and Ratner, M. and Serre, D. and Sinai, {\relax Ya}. G. and Sloane, N. J. A. and Tits, J. and Waldschmidt, M. and Watanabe, S. and Berger, M. and Coates, J. and Varadhan, S. R. S.},
  year = 1999,
  series = {Grundlehren Der Mathematischen {{Wissenschaften}}},
  volume = {319},
  publisher = {Springer Berlin Heidelberg},
  address = {Berlin, Heidelberg},
  doi = {10.1007/978-3-662-12494-9},
  urldate = {2026-03-13},
  copyright = {https://www.springernature.com/gp/researchers/text-and-data-mining},
  isbn = {978-3-642-08399-0 978-3-662-12494-9},
  langid = {english}
}

@misc{pozzettiFinslerMetrics$12026,
  title = {Finsler Metrics on \$1/N\$-Translation Structures on Surfaces},
  author = {Pozzetti, Beatrice and Shi, Jiajun},
  year = 2026,
  month = apr,
  number = {arXiv:2604.02243},
  eprint = {2604.02243},
  primaryclass = {math},
  publisher = {arXiv},
  doi = {10.48550/arXiv.2604.02243},
  urldate = {2026-04-22},
  archiveprefix = {arXiv}
}

@article{vianaDynamicsIntervalExchange,
  title = {Dynamics of {{Interval Exchange Transformations}} and {{Teichmu}}\textasciidieresis ller {{Flows}}},
  author = {Viana, Marcelo},
  langid = {english}
}

@article{reesAlternativeApproachErgodic1981,
  title = {An Alternative Approach to the Ergodic Theory of Measured Foliations on Surfaces},
  author = {Rees, Mary},
  year = 1981,
  month = dec,
  journal = {Ergodic Theory and Dynamical Systems},
  volume = {1},
  number = {4},
  pages = {461--488},
  issn = {0143-3857, 1469-4417},
  doi = {10.1017/S0143385700001383},
  urldate = {2026-05-13},
  copyright = {https://www.cambridge.org/core/terms},
  langid = {english}
}

@article{kerckhoffErgodicityBilliardFlows1986,
  title = {Ergodicity of {{Billiard Flows}} and {{Quadratic Differentials}}},
  author = {Kerckhoff, Steven and Masur, Howard and Smillie, John},
  year = 1986,
  month = sep,
  journal = {The Annals of Mathematics},
  volume = {124},
  number = {2},
  eprint = {1971280},
  eprinttype = {jstor},
  pages = {293},
  issn = {0003486X},
  doi = {10.2307/1971280},
  urldate = {2025-09-15},
  langid = {english}
}

@article{constantineMarkedLengthSpectrum2018,
  title = {Marked {{Length Spectrum Rigidity}} in {{Non-Positive Curvature}} with {{Singularities}}},
  author = {Constantine, David},
  year = 2018,
  journal = {Indiana University Mathematics Journal},
  volume = {67},
  number = {6},
  eprint = {45010366},
  eprinttype = {jstor},
  pages = {2337--2361},
  publisher = {Indiana University Mathematics Department},
  issn = {0022-2518},
  urldate = {2026-06-23}
}

@book{pennerCombinatoricsTrainTracks1992,
  title = {Combinatorics of Train Tracks},
  author = {Penner, Robert C. and Harer, John L.},
  year = 1992,
  series = {Annals of Mathematics Studies},
  number = {no. 125},
  publisher = {Princeton University Press},
  address = {Princeton, N.J},
  isbn = {978-0-691-08764-1 978-0-691-02531-5 978-1-4008-8245-8},
  langid = {english}
}

@article{boissyDynamicsGeometryRauzy2009,
  title = {Dynamics and Geometry of the {{Rauzy}}--{{Veech}} Induction for Quadratic Differentials},
  author = {Boissy, Corentin and Lanneau, Erwan},
  year = 2009,
  month = jun,
  journal = {Ergodic Theory and Dynamical Systems},
  volume = {29},
  number = {3},
  pages = {767--816},
  issn = {0143-3857, 1469-4417},
  doi = {10.1017/S0143385708080565},
  urldate = {2026-06-25},
  copyright = {https://www.cambridge.org/core/terms},
  langid = {english}
}

@article{bonahonBoutsVarietesHyperboliques1986,
  title = {Bouts Des {{Vari\'et\'es Hyperboliques}} de {{Dimension}} 3},
  author = {Bonahon, Francis},
  year = 1986,
  journal = {Annals of Mathematics},
  volume = {124},
  number = {1},
  eprint = {1971388},
  eprinttype = {jstor},
  pages = {71--158},
  publisher = {[Annals of Mathematics, Trustees of Princeton University on Behalf of the Annals of Mathematics, Mathematics Department, Princeton University]},
  issn = {0003-486X},
  doi = {10.2307/1971388},
  urldate = {2026-06-29}
}

@article{kahnConformalSurfaceEmbeddings2022,
  title = {Conformal Surface Embeddings and Extremal Length},
  author = {Kahn, Jeremy and Pilgrim, Kevin M. and Thurston, Dylan P.},
  year = 2022,
  month = sep,
  journal = {Groups, Geometry, and Dynamics},
  volume = {16},
  number = {2},
  pages = {403--435},
  issn = {1661-7207, 1661-7215},
  doi = {10.4171/ggd/673},
  urldate = {2026-06-04}
}

@article{sasaki2022currents,
  title={Currents on cusped hyperbolic surfaces and denseness property.},
  author={Sasaki, Dounnu},
  journal={Groups, Geometry \& Dynamics},
  volume={16},
  number={3},
  year={2022}
}

@book{erlandssonMirzakhanisCurveCounting2022,
  title = {Mirzakhani's {{Curve Counting}} and {{Geodesic Currents}}},
  author = {Erlandsson, Viveka and Souto, Juan},
  year = 2022,
  series = {Progress in {{Mathematics}}},
  volume = {345},
  publisher = {Springer International Publishing},
  address = {Cham},
  doi = {10.1007/978-3-031-08705-9},
  urldate = {2025-09-15},
  copyright = {https://www.springer.com/tdm},
  isbn = {978-3-031-08704-2 978-3-031-08705-9},
  langid = {english}
}
\end{document}